\documentclass[a4paper]{amsart}

\usepackage[utf8]{inputenc}

\usepackage[american]{babel}
\usepackage{amsthm}
\usepackage{amssymb}
\usepackage{amsmath}
\usepackage[shortlabels]{enumitem}
\usepackage{mathrsfs}
\usepackage{mathtools}
\usepackage{url}

\usepackage{tikz} 
\usetikzlibrary{matrix,arrows,calc,decorations.pathreplacing,fit}
\usepackage{tikz-cd}

\numberwithin{equation}{section}

\newtheorem{theorem}{Theorem}[section]
\newtheorem{lemma}[theorem]{Lemma}
\newtheorem{corollary}[theorem]{Corollary}
\newtheorem{proposition}[theorem]{Proposition}

\theoremstyle{definition}
\newtheorem{definition}[theorem]{Definition}
\newtheorem{construction}[theorem]{Construction}
\newtheorem{notation}{Notation}
\theoremstyle{remark}
\newtheorem{remark}[theorem]{Remark}
\newtheorem{example}[theorem]{Example}

\DeclareMathOperator{\Ad}{Ad}

\DeclareMathOperator{\codim}{codim}

\DeclareMathOperator{\CRYS}{CRIS}
\DeclareMathOperator{\defect}{def}

\DeclareMathOperator{\Fil}{Fil}

\DeclareMathOperator{\id}{id}
\DeclareMathOperator{\image}{im}

\DeclareMathOperator{\Gal}{Gal}
\DeclareMathOperator{\GL}{GL}
\DeclareMathOperator{\GSp}{GSp}

\DeclareMathOperator{\Hom}{Hom}

\DeclareMathOperator{\Isom}{Isom}

\DeclareMathOperator{\Norm}{Norm}

\DeclareMathOperator{\pr}{pr}
\DeclareMathOperator{\Rep}{Rep}

\DeclareMathOperator{\Sh}{Sh}

\DeclareMathOperator{\Spec}{Spec}

\DeclareMathOperator{\rank}{rk}

\DeclareMathOperator{\val}{val}

\def \ad {\mathrm{ad}}

\def \dom {\mathrm{dom}}

\newcommand{\pot}[1]{ [\hspace{-0,17em}[ {#1} ]\hspace{-0,17em}] }

\def \AA {\mathbb{A}}

\def \DD {\mathbb{D}}

\def \FF {\mathbb{F}}
\def \GG {\mathbb{G}}

\def \QQ {\mathbb{Q}}
\def \RR {\mathbb{R}}

\def \ZZ {\mathbb{Z}}

\def \Acal {\mathcal{A}}

\def \Fcal {\mathcal{F}}
\def \Gcal {\mathcal{G}}
\def \Hcal {\mathcal{H}}

\def \Ocal {\mathcal{O}}

\def \Xcal {\mathcal{X}}
\def \Ycal {\mathcal{Y}}

\def \Cfr {\mathfrak{C}}

\def \Ifr {\mathfrak{I}}

\def \Nfr {\mathfrak{N}}

\def \Pfr {\mathfrak{P}}

\def \Sfr {\mathfrak{S}}
\def \Tfr {\mathfrak{T}}

\def \Iscr {\mathscr{I}}

\def \Mscr {\mathscr{M}}

\def \Sscr {\mathscr{S}}

\def \hbar {\bar{h}}

\def \sbar {\bar{s}}

\def \Gsf {\mathbf{G}}

\def \Mbf {\mathbf{M}}

\def \Pbf {\mathbf{P}}

\def \bbf {\mathbf{b}}

\def \Gsf  {\mathsf{G}}

\def \FFbar {\overline{\mathbb{F}}}

\def \QQbar {\overline{\mathbb{Q}}}

\def \BT    {{Barsotti-Tate group}}

\def \Lra   {\Leftrightarrow}

\def \mono  {\hookrightarrow}

\def \epi   {\twoheadrightarrow}

\def \isom  {\stackrel{\sim}{\rightarrow}}

\newcommand{\bigslant}[2]{{\raisebox{.2em}{$#1$}\hspace{-.3em}\left/ \hspace{-.2em}\raisebox{-.2em}{$#2$}\right.}}

\def \Xund {\underline{X}}

\def \sund {\underline{s}}
\def \tund {\underline{t}}
\def \uund {\underline{u}}

\def \univ {{\rm univ}}

\def \Def {{\mathfrak{Def}}}

\def \BTT {{Barsotti-Tate group with crystalline Tate tensors}}
\def \BTTs {{Barsotti-Tate groups with crystalline Tate tensors}}

\newcounter{subenvcounter}
\newenvironment{subenv}{%
 \begin{list}
  {\em (\arabic{subenvcounter})}
  {\setlength{\leftmargin}{20pt}
   \setlength{\rightmargin}{0pt}
   \setlength{\itemindent}{0pt}
   \setlength{\labelsep}{5pt}
   \setlength{\labelwidth}{13pt}
   \setlength{\listparindent}{\parindent}
   \setlength{\parsep}{0pt}
   \setlength{\itemsep}{0pt}
   \setlength{\topsep}{-\parskip}
   \usecounter{subenvcounter}}}
  {\end{list}}

\newcounter{asslistcounter}
\newenvironment{assertionlist}{
 \begin{list}
  {\upshape (\alph{asslistcounter})}
  {\setlength{\leftmargin}{18pt}
   \setlength{\rightmargin}{0pt}
   \setlength{\itemindent}{0pt}
   \setlength{\labelsep}{5pt}
   \setlength{\labelwidth}{13pt}
   \setlength{\listparindent}{\parindent}
   \setlength{\parsep}{0pt}
   \setlength{\itemsep}{0pt}
   \setlength{\topsep}{-.5\parskip}
   \usecounter{asslistcounter}}}
  {\end{list}}

\newenvironment{bulletlist}{
 \begin{list}
  {$\bullet$}
  {\setlength{\leftmargin}{18pt}
   \setlength{\rightmargin}{0pt}
   \setlength{\itemindent}{0pt}
   \setlength{\labelsep}{5pt}
   \setlength{\labelwidth}{13pt}
   \setlength{\listparindent}{\parindent}
   \setlength{\parsep}{0pt}
   \setlength{\itemsep}{0pt}
   \setlength{\topsep}{-.5\parskip}
   \usecounter{bulllistcounter}}}
 {\end{list}}

\author[P.~Hamacher]{Paul~Hamacher}
 
\title[The almost product structure of Newton strata ]{The almost product structure of Newton strata in the Deformation space of a Barsotti-Tate group with crystalline Tate tensors}

\address{Technische Universi\"at M\"unchen \\ Zentrum Mathematik - M11 \\ Boltzmannstra{\ss}e 3 \\ 85748 Garching bei M\"unchen \\ Deutschland}

\email{hamacher@ma.tum.de}
\begin{document}

 \begin{abstract}
  In this paper, we construct the almost product structure of the minimal Newton stratum in deformation spaces of {\BTTs}, similar to Oort's and Mantovan's construction for Shimura varieties of PEL-type. It allows us to describe the geometry of the Newton stratum in terms of the geometry of two simpler objects, the central leaf and the isogeny leaf. This yields the dimension and the closure relations of the Newton strata in the deformation space. In particular, their nonemptiness shows that a generalisation of Grothendieck's conjecture of deformations of {\BT}s with given Newton polygon holds.
  
   As  an application, we determine analogous geometric properties of the Newton stratification of Shimura varieties of Hodge type and prove the equidimensionality of Rapoport-Zink spaces of Hodge type. 
 \end{abstract} 

 \maketitle
 
 \section{Introduction}
 
 Throughout the paper, let $k$ be an algebraically closed field of characteristic $p > 2$. We denote by $W$ the its ring of Witt vectors, $L := W[\frac{1}{p}]$ and by $\sigma$ the Frobenius on $K, W$ or $k$. 
                                                                                                                                                                                                                                                                                                                                                                                                                                                                                                                                                                                                                                                                                                                                                                                                                                                                                                                                                                                                                                                                                                                                                                                                                                                                                                                      
 Let $G$ be a reductive group over $\ZZ_p$, $\bbf$ a $\sigma$-conjugacy class in $G(\widehat\QQ_p^{\rm nr})$ and $\mu$ a cocharacter of $G$ such that $(G,b,-\mu)$ is an integral local Shimura datum of Hodge type in the sense of \cite{kim}~Def.~2.5.10. We fix a faithful representation $i: G \mono \GL(M)$ such that $\mu$ acts with weights $0$ and $1$ on $M$ and a family $\sund$ of tensors of $M$ such that $G$ is the stabiliser of $\sund$. Let $(X,\tund)$ be a {\BTT} over  $k$ such that there exists an isomorphism of the Dieudonn\'e module of $X$ with $M \otimes_{\ZZ_p} W$ identifying its Frobenius with an element of the $\sigma$-conjugacy class $i(\bbf)$ and the crystalline Tate tensors with $\sund \otimes 1$. We denote by $\Def_G(X)$ the formal subscheme of the  deformation space $\Def(X)$ of $X$ cut out by $i$, whose power series-valued points can be regarded as those deformations where the crystalline Tate tensors $\tund$ extend to the whole group (\cite{faltings99}, \cite{moonen98}). The space $\Def_G(X)$ describes the local geometry of the Rapoport-Zink space associated to $(G,\bbf,-\mu)$ defined by Kim in \cite{kim} (see also \cite{HP} for an alternative definition), and in the case where $(G,\bbf,-\mu)$ is induced by a Shimura datum of Hodge type also the local geometry of the canonical integral model of the corresponding Shimura variety (\cite{kisin10}~\S~2.3). 
 
 \subsection{Results on deformation spaces}
 From now on we change the notation as follows. Let $\Def(X)$ and $\Def_G(X)$ denote the algebraisations of the special fibres of the respective formal schemes. The ``universal deformation'' of $(\Xund,\tund)$ algebraises to a {\BTT} $(\Xcal^\univ \tund^\univ)$ over $\Def_G(X)$ by Messing's algebraisation result for {\BT}s \cite{messing72}~Lemma~II.4.16 and de Jong's calculations on Dieudonn\'e crystals in the proof of \cite{dJ98} Prop.-~2.4.8 (cf. \cite{HP} Rem.~2.3.5~c).
 
 The Newton stratification is defined as the stratification corresponding to the isogeny class of the fibre of $(\Xcal^\univ, \tund^\univ)$ over the geometric points of $\Def_G(X)$. By Dieudonn\'e theory the isogeny classes correspond to a certain finite subset $B(G,\mu)$ of the set $B(G)$ of $\sigma$-conjugacy classes in $G(L)$. For $\bbf \in B(G,\mu)$, denote by $\Def_G(X)^\bbf$ the corresponding Newton stratum.

 Denote by $\bbf_0$ the isogeny class of $(X,\sund)$ and by $\Nfr_G(X) := \Def_G(X)^{\bbf_0}$ the (unique) minimal Newton stratum. We construct a surjective inseparable finite-to-finite correspondence between $\Nfr_G(X)$ and the product of its central leaf and its isogeny leaf. Here the central leaf $\Cfr_G(X) \subset \Nfr_G(X)$ is defined as the locus where the fibre of $(\Xcal^\univ, \tund^\univ)$ over  the geometric points is \emph{isomorphic} to $(X,\tund)$ and the isogeny leaf $\Ifr_G(X)$ is the maximal reduced subscheme of $\Nfr_G(X)$ such that the restriction $(\Xcal^\univ, \tund^\univ)_{| \Ifr_G(X)}$ is isogenous to $(X,\tund)_{\Ifr_G(X)}$. If $(G,\bbf,-\mu)$ is of PEL-type and $X$ is minimal, then this almost product structure is nothing else than the restriction of the almost product structure of Newton strata in Shimura varieties of PEL type constructed by Mantovan in \cite{mantovan05}. 
 
 Following Lovering's construction in \cite{lovering}~\S~3, we obtain an $F$-isocrystal with $G_{\QQ_p}$-structure over $\Def_G$. After generalising Yangs purity result for $F$-isocrystals (\cite{yang11} Thm.~1.1) to $F$-isocrystals with $G_{\QQ_p}$-structure, the geometric properties of the Newton stratification follow by a purely combinatorial argument (see e.g.\ \cite{viehmann}~Lemma~5.12).
 
 \begin{theorem} \label{thm main}
  Let $\bbf \in B(G,\mu)$ such that $\bbf \geq \bbf_0$ with respect to the partial order on $B(G)$. Then
  \begin{subenv}
   \item $\Def_G(X)^\bbf$ is non-empty and of pure dimension $\langle \rho, \mu + \overline{\nu}(\bbf)\rangle - \frac{1}{2} \defect(\bbf)$, where $\rho$ denotes the halfsum of positive roots of $G$, $\overline{\nu}(\bbf)$ denotes the Newton point of $\bbf$ and $\defect(\bbf)$ denotes the defect of $\bbf$.
   \item $\overline{\Def_G(X)^\bbf} = \bigcup_{\bbf' \leq \bbf} \Def_G(X)^{\bbf'}$.
  \end{subenv}
 \end{theorem}
 In the cases where $X$ is without additional structure, or when the additional structure is given by a polarisation (i.e.\ $G=\GL_n$ or $G=\GSp_{n}$), this result was proven by Oort (\cite{oort00} Th.~3.2,~3.3). If the additional structure is of PEL-type, the above theorem was proven in a earlier work of this author (\cite{hamacher15b}) invoking more complicated results on Shimura varieties.
 
 \subsection{Application to Shimura varieties and Rapoport-Zink spaces} 
   
  Let $(\Gsf,X)$ be a Shimura datum of Hodge-type, $K \subset G(\AA_f)$ a small enough compact open subgroup which is hyperspecial at $p$. Denote by $G$ the reductive model of $\Gsf$ corresponding to $K_p$. The existence of a canonical integral model $\Sscr_G$ of the Shimura variety $\Sh_K(\Gsf,X)$ was shown by Kisin \cite{kisin10}, with some exceptions in the case $p=2$. His construction of $\Sscr_G$ also equips it with an Abelian scheme $\Acal_G \to \Sscr_G$ and crystalline Tate tensors $\tund_{G,x}$ on $\DD(\Acal_{G,x}[p^\infty])$ for every point $x \in \Sscr_G(\FFbar_p)$. While the {\BTT} $(\Acal_{G,x}[p^\infty], \tund_x)$ depends on some choices made during the construction of $\Sscr_G$, it induces an isocrystal with $G$-structure $\Gcal_x$ over $\FFbar_p$ which is independent of them. Lovering constructed an isocrystal with $G$-structure over the special fiber $\Sscr_{G,0}$ of $\Sscr_G$, which specialises to $\Gcal_x$ for every $x \in \Sscr_G(\FFbar_p)$. Thus we have a well-defined Newton stratification on $\Sscr_{G,0}$; we denote the Newton strata by $\Sscr_{G,0}^\bbf$.
  
  As the isomorphism of the formal neighbourhood $\Sscr_{G,0,x}^\wedge$ with $\Def_G(\Acal_{G,x}[p^\infty])$ given in \cite{kisin10} \S~2.3 respects the Newton stratification, the following is a direct consequence of Theorem~\ref{thm main}
  \begin{theorem}
   Assume that $\Sscr_G^\bbf$ is non-empty. Then
   \begin{subenv}
    \item $\Sscr_G^\bbf$ is of pure dimension $\langle \rho, \mu + \overline{\nu}(\bbf)\rangle - \frac{1}{2} \defect(\bbf)$.
    \item $\overline{\Sscr_G^\bbf} = \bigcup_{\bbf' \leq \bbf} \Sscr_G^{\bbf'}$.
   \end{subenv}
  \end{theorem}
 
  Another interesting consequence of the almost product structure is the equidimensionality of Rapoport-Zink spaces, which was conjectured by Rapoport in \cite{rapoport05}. It was proven by Viehmann in the cases  $G=\GL_n$ in \cite{viehmann08a} and $G=\GSp_{n}$ in \cite{viehmann08b}. Our precise statement is as follows.
  
  \begin{theorem}
   Any Rapoport-Zink space of Hodge type coming from an unramified local Shimura datum as in \cite{kim} is equidimensional.
  \end{theorem}

 This is proven by a similar method as in \cite{HV12} where the function field analogue is shown.

 \subsection{Overview}
 We first recall basic definitions and notation needed for the construction of the almost product structure in section~\ref{sect preliminaries}. Moreover we generalise various results known for \BT s to {\BTTs}, i.~a.\ Yang's purity theorem. In section \ref{sect leaves} we discuss geometric properties of the central and isogeny leaves as well as the restiction of the universal deformation to them before defining the almost product structure in section~\ref{sect almost product local}. Finally, we deduce the above theorems in section~\ref{sect dimension formulas}.

 This article was written independently from the article \cite{zhang} of Zhang, which was put on Arxiv approximately two weeks prior to the first version of this article. In his work, he also calculates the dimension of central leaves and Newton strata in $\Sscr_G^\bbf$. In contrast to my work, he does not consider the local geometry of the Shimura variety, but uses the geometry of stratifications of $\Sscr_G^\bbf$ directly.
  
 \emph{Acknowledgements:} I am grateful to Mark Kisin for many helpful discussions and his advice. I warmly thank Thomas Lovering for giving me a preliminary version of his thesis. I thank Stephan Neupert and Eva Viehmann for pointing out some mistakes in the preliminary version of this article. Most of this work was written during a stay at the Harvard University which was supported by a fellowship within the Postdoc program of the German Academic Exchange Service (DAAD). I want to thank the Harvard University for its hospitality. Moreover, the author was partially supported by the ERC starting grant 277889 ``Moduli spaces of local $G$-shtukas''.
  
 \section{Preliminaries} \label{sect preliminaries}
 
  \subsection{The Newton stratification}
 
Let $H$ be a reductive group over $\QQ_p$. We fix a quasi-split inner form $H_0$ togehter with an $H^\ad(\overline{L})$-conjugacy class of isomorphisms $\psi:H_{0,\overline{L}} \isom H_{\overline{L}}$. We fix  $S_0 \subset T_0 \subset B_0 \subset H_0$ where $S_0$ is a maximal split torus, $T_0$ a maximal torus and $B_0$ a Borel subgroup of $H_0$. Denote by $\pi_1(H)$ the  fundamental group of $H$, that is the quotient of the absolute cocharacter lattice by the coroot lattice. This construction can be made independent of the choice of a maximal torus (cf. \cite{RR96}~\S~1).

There is a canonical bijection $\bbf$ between the isomorphism classes of isocrystals with $H$-structure over $k$ and the set of $\sigma$-conjugacy classes. We briefly recall the classification of $\sigma$-conjugacy classes from \cite{kottwitz85}, \cite{kottwitz97} as presented in \cite{RR96}~\S~1.

 For an element $b\in H(L)$ denote by $[b] := \{gb\sigma(b)^{-1} \mid g \in G(L)\}$ the corresponding $\sigma$-conjugacy class and by $B(G)$ the set of $\sigma$-conjugacy classes of $G(L)$; this definition is independent of $k$ by \cite{RR96}~Thm.~1.1. Kottwitz classifies the set by the two invariants $\overline{\nu}$ and $\kappa$   in \cite{kottwitz97} \S~4.13. To define $\overline{\nu}$, denote by $D$ the pro-torus over $\QQ_p$ with character group $\QQ$. Kottwitz defines the slope morphism $\nu: H(L) \to \Hom(D_L,G_L)$ as the unique functorial morphism  such that for $H = \GL(V)$ the quasi-cocharacter $\nu(b)$ induces the slope decomposition of the isocrystal $(V,b\sigma)$. Since $\nu(gb\sigma(g)^{-1}) = Int(g) \circ \nu(b)$ the slope homomorphism induces a $G(L)$-conjugacy class $\overline{\nu}([b])$ in $\Hom(D_L,H_L)$, which is called the Newton point. It can be shown that $\nu([b])$ is invariant under $\sigma$. Rather than working with orbits, we regard $\overline{\nu}([b]) \in X_\ast(S_0)_{\QQ,\dom}$ as the unique dominant element in $\psi^{-1} \circ \nu([b])$. The Kottwitz point $\kappa: B(H) \to \pi_1(H)_{\Gal(\QQbar_p/\QQ_p)}$  is defined as the (unique) transformation of functors such that for $H=\GG_m$ it is induced by $\val_p: \GG_m(\QQ_p) \to \ZZ \cong \pi_1(\GG_m)$.  These invariants define a partial order on $B(G)$ given by $\bbf' \leq \bbf$ iff $\overline{\nu}(\bbf') \leq \overline{\nu}(\bbf)$ with respect to the dominance order and $\kappa(\bbf') = \kappa(\bbf)$. 
 
 We recall the basics of the Newton stratification. All statements in this paragraph are proven in \cite{RR96}~\S~3. If we are given an isocrystal $\Fcal$ with $G$-structure over a Noetherian $k$-scheme $S$, we denote
 \begin{eqnarray*}
  S^\bbf &=& \{ s \in S \mid \bbf(\Fcal_{\sbar}) = \bbf\},
 \end{eqnarray*}
 where $\sbar$ is an arbitrary geometric point over $s$. We define $S^{\leq \bbf}$ and $S^{<\bbf}$ analogously. The decomposition of $S$ as a union of $S^\bbf$ is called the Newton stratification. This is in general not a stratification in the strict sense, i.e.\ the closure of $S^\bbf$ does not need to be a union of strata. However, one has that $S^{\leq \bbf}$ is closed; in particular the $S^\bbf$ are locally closed and can be regarded as reduced subschemes.

 An important geometric property of the Newton stratification is its purity. We distinguish between the following notions of purity.
 
 \begin{definition}
  Let $\Fcal$ be an isocrystal with $H$-stucture over an $\FF_p$-scheme $S$, where $H$ is a reductive group over $\QQ_p$. We say that the Newton stratification of $S$ satisfies
  \begin{subenv}
   \item weak purity, if the embedding $S^\bbf \mono S^{\leq \bbf}$ is affine for every $\bbf \in B(H)$.
   \item strong purity, if the embedding $S^{\leq \bbf} \setminus S^{\leq \bbf'} \mono S^{\leq\bbf}$ is affine for every $\bbf \in B(G)$ and maximal element $\bbf' \in B(H)_{<\bbf}$. 
   \item topological strong purity, if for every $\bbf \in B(H)$ and maximal element $\bbf' \in B(H)_{<\bbf}$ and every irreducible component $I \subset S^{\leq \bbf}$ the subscheme $I^{\leq \bbf'} \subseteq I$ is either empty or pure of codimension $\leq 1$.
  \end{subenv}
 \end{definition}

 It was shown by Vasiu in \cite{vasiu06} that the Newton stratification associated to an $F$-isocrystal without additional structure satisfies weak purity. By evaluating $\Fcal$ at a familiy of representations of $H$ which seperates Newton points, one easily deduces that weak purity also holds for $F$-isocrystals with $H$-structure.

 It is not known whether the Newton stratification associated to an $F$-isocrystal without additional structure satisfies strong purity, but Yang's improvement of de Jong-Oort's purity theorem says that it satisfies topological strong purity if the base scheme is Noetherian (\cite{yang11}~Thm.~1.1).
  
 As topological strong purity will be needed in section~\ref{sect dimension formulas}, we prove the following proposition. 
  
 \begin{proposition} \label{prop purity}
  Let $S$ be an $\FF_p$-scheme such that every $F$-isocrystal over $S$ the Newton stratification satisfies strong purity (resp.\ topological strong purity). Then for any isocrystal with $H$-structure over $S$, the Newton stratification satisfies strong purity (resp.\ topological strong purity).
 \end{proposition} 

 The proof will use Viehmann's description of $B(H)_{\leq \bbf'}$, for which we need the following definition.
 
 \begin{definition}
 We consider $\nu \in X_\ast(S_0)_{\QQ, \dom}$ and a simple positive relative root $\beta$ of $H_0$.
  \begin{subenv}
   \item We say that $\nu$ has a break point at $\beta$ if $\langle \beta, \nu \rangle > 0$.
   \item The value of $\nu$ at $\beta$ is given by $ \pr_\beta(\nu) := \langle\omega_{\beta^\vee},\nu\rangle $, where $\omega_{\beta^\vee}$ denotes the corresponding relative fundamental weight of $H_{0,\ad}$. 
  \end{subenv}
 \end{definition}
 
 \begin{example}
 To motivate the definition of break points, consider $H = H_0 = \GL_n$ with $B_0$ the Borel of upper triangular matrices and $S_0=T_0$ the diagonal torus. We have the standard identification $X_\ast(S_0)_{\QQ} = \QQ^n$ such that $\nu \in \QQ^n$ is dominant if $\nu_1 \geq \nu_2 \geq \ldots \geq \nu_n$. The simple roots $\alpha_i$ of $(G,B,T)$ are given by $\langle \alpha_i, \nu \rangle = \nu_i-\nu_{i+1}$.
For simplicity we denote $\pr_{\alpha_i}$ by $\pr_i$.  
  
  Let $P_\nu$ be the polygon associated to $\nu$, that is the graph of the piecewise linear function $f_\nu$ over $[0,n]$ with $f(0) = 0$ and slope $\nu_i$ over $(i-1,i)$. Then the breakpoints of $P_\nu$ are precisely at the $x$-coordinates $i$ such that $\nu_i \not= \nu_{i+1}$, that is $\langle \alpha_i, \nu \rangle >0$. Note that the $y$-coordinate of $f$ at $i$ is given by
 \[
  f_\nu(i) = \nu_1 + \ldots + \nu_i = \pr_i(\nu) - \frac{i}{n} \cdot f_\nu(n).
 \]
 So $f_\nu(i)$ and $\pr_i(\nu)$ differ only by a constant depending on $i$ and $f(n)$. In particular for $\nu,\nu ' \in \QQ^n$ we have $f_{\nu '}(i) < f_{\nu}(i)$ if and only if $\pr_i(\nu ') < \pr_i(\nu)$.
 \end{example}

 \begin{lemma}[\cite{viehmann13}~Lemma~5, Rem.~6]
  There exists a canonical bijection between the maximal elements of $B(H)_{< \bbf}$ and the break points of $\overline{\nu}(\bbf)$. If $\bbf'$ corresponds to a break point at $\beta$, we have
  \[
   B(H)_{\leq \bbf'} = \{\bbf'' \in B(H) \mid \bbf'' < \bbf, \pr_{\beta}(\overline{\nu}(\bbf'')) < \pr_{\beta}(\overline{\nu}(\bbf)) \}
  \]
 \end{lemma}

 \begin{remark}
  Technically Chai's article \cite{chai00}, whose calculations are used in above lemma only considers the case where $H$ is quasi-split. But as Chai remarked, this assumption is only for simplicity. Indeed, we can replace $H$ by $H_\ad$ in above lemma and thus assume that $H$ is adjoint. Then by \cite{kottwitz97} p.~295f.\ there exists a bijection $B(H) \to B(H_0)$ which is the identity on Newton points and induces a bijection on Kottwitz points. 
 \end{remark}

 \begin{proof}[Proof of Proposition~\ref{prop purity}]
  None of the above notions change when we replace $\Fcal$ by the induced isocrystal with $H_\ad$-structure; thus we may assume that $H$ is semisimple. 
  
 We prove the theorem by constructing a suitable representation $\rho: H \to \GL_n$ and applying the postulated strong purity of the Newton stratification of $\Fcal(\rho)$ to a suitable breakpoint. Thus it suffices to show that for every simple relative root $\beta$ of $H$ there exists a representation $\rho: H \to \GL_n$ and an integer $1 \leq i \leq n-1$ such that the following conditions are satisfied.
  \begin{assertionlist}
   \item For any $\nu \in X_\ast(S)_{\QQ,\dom}$ with a break point at $\beta$, the quasi-cocharacter $\rho(\nu)_\dom$ has a break point at $i$.
   \item For $\nu' \leq \nu$ one has $\pr_\beta (\nu') < \pr_\beta (\nu)$ if and only if $\pr_i (\rho(\nu')_\dom) < \pr_\beta (\rho(\nu)_\dom)$.
  \end{assertionlist} 
 Let $\omega_\beta$ be the fundamental weight associated to $\beta^\vee$ and $N$ be a positive integer such that $N \omega_\beta \in X^\ast(S_0)$. Now let $\rho$ be any representation such that the restriction to $S_{0\, L}$ has unique highest weight $N\omega_\beta$, and let $i$ be the multiplicity of the weight $N\omega_\beta$ in $\rho$. Then
 \[
  \rho(\nu)_\dom =(\underbrace{N\cdot\pr_\beta(\nu), \ldots, N\cdot\pr_\beta(\nu)}_{i},\langle \lambda, \nu \rangle, \ldots),
 \]
 for some weight $\lambda \not= N\omega_\beta$ of $\rho$.
 
 Thus $\rho$ and $i$ satisfy (b) and $\rho(\nu)_\dom$ has a break point at $i$ iff $\langle N \omega_\beta - \lambda, \nu \rangle > 0$. As
 \[
  \langle N\omega_\beta - \lambda, \nu \rangle > \langle N\omega_\beta - \lambda, \langle \beta,\nu \rangle \cdot \omega_{\beta}^\vee \rangle = \langle \beta,\nu \rangle \cdot \langle N\omega_\beta - \lambda, \omega_\beta^\vee\rangle,
 \]
 it suffices to show that $\langle N\omega_\beta - \lambda, \omega_\beta^\vee\rangle > 0$. Assume the contrary, i.e.\ $\langle N\omega_\beta - \lambda, \omega_\beta^\vee \rangle =0$, or equivalently that $N\omega_\beta -\lambda$ is orthogonal to $N\omega_\beta$ with respect to a scalar product $(\cdot \mid \cdot)$ on $X^\ast(S)_\RR$ which is invariant under the Weyl group. In particular, we would that get $(\lambda | \lambda) > (N\omega_\beta | N\omega_\beta)$. But since the set of weights of $\rho$ is stable under the Weyl group of $G$, $\lambda$ must be contained in the convex hull of the Weyl group orbit of $N\omega_\beta$ (\cite{RR96}~Lemma~2.2). Hence we get $(\lambda | \lambda)\leq (N\omega_\beta | N\omega_\beta)$. Contradiction.
 \end{proof}
 
 \begin{corollary} \label{cor purity}
  Let $\Fcal$ be an isocrystal with $H$-structure over a Noetherian $\FF_p$-scheme $S$. Then the Newton stratification on $S$ satisfies topological strong purity.
 \end{corollary}  
 
 \subsection{\BTTs} \label{ss BTTs}
  
 We recall the definition of crystalline Tate tensors as given by Kim in \cite{kim}. For any object $M$ of a rigid quasi-Abelian tensor category, we denote by $M^\otimes$ the direct sum of any finite combination of tensor products, symmetric products, alternating products and duals of $M$. We define a \emph{tensor of $M$} to be a morphism $s: \mathbf{1} \to M^\otimes$, where $\mathbf{1}$ denotes the unit object.
 
  For any {\BT} $\Xcal$ over a formally smooth $\FF_p$-scheme $S$ denote by $\DD(\Xcal)$ its contravariant Dieudonn\'e crystal as in \cite{BBM82}~\S~3.3. It is a locally free crystal of $\mathcal{O}_{S/\ZZ_p,\CRYS}$-module of rank equal to the height of $\Xcal$. The pull-back $\DD(\Xcal)_S$ to the Zariski site of $S$ is naturally equipped with the Hodge filtration $\Fil_{\DD(\Xcal)}^1 := \omega_\Xcal \subset \DD(\Xcal)_S$, which is Zariski-locally a direct summand of rank $\dim \Xcal$. Moreover the relative Frobenius of $\Xcal$ over $S$ induces a map $F: \DD(\Xcal)^{(p)} \to \DD(\Xcal)$, which is also called the Frobenius. As the relative Frobenius is an isogeny, $F$ isnduces an isomorphism of isocrystals $\DD(\Xcal)[\frac{1}{p}]^{(p)} \isom \DD(\Xcal)[\frac{1}{p}]$. The Frobenius induces the Hodge filtration via
 \begin{equation} \label{eq Hodge Filtration vs Frobenius}
  {(\Fil_{\DD(\Xcal)}^1)}^{(p)} = \ker F_{|\DD(\Xcal)_S}.
 \end{equation}
  This additional structure induces a filtration $\Fil_{\DD(\Xcal)^\otimes}^\bullet$ of $\DD(\Xcal)^\otimes_S$ and a morphism of \emph{iso}crystals $F: \DD(\Xcal)[\frac{1}{p}]^{\otimes\, (p)} \to \DD(\Xcal)[\frac{1}{p}]^\otimes$.
 
 \begin{definition}
  Let $\Xcal$ be a {\BT} over a formally smooth scheme $S$ in characteristic $p$. A tensor $t$ of $\DD(\Xcal)$ is called a crystalline Tate tensor if it induces a morphism of $F$-isocrystals $\mathbf{1} \to \DD(\Xcal)[\frac{1}{p}]^\otimes$, i.e. it is Frobenius equivariant.
 \end{definition}
 
 \begin{remark}
  It follows from equation~\ref{eq Hodge Filtration vs Frobenius} and the Frobenius equivariance that any crystalline Tate tensor is contained in $\Fil_{\DD(\Xcal)^\otimes}^0$.
 \end{remark}

 We fix a reductive group scheme $G$ over $\ZZ_p$ and let $M$ be faithful representation of $G$ of finite rank. By \cite{kim}~Prop.~2.1.3 (see \cite{kisin10} Prop.1.3.2 for the original statement) there exists a family of tensors $\sund \subset M^\otimes$ such that $G$ equals the stabilizer of $M$. We fix such a choice of $(M,\sund)$.
 
 \begin{definition}
  Let $\Xcal$ be a {\BT} over a formally smooth scheme $S$ in characteristic $p$. An $\sund$-structure on $\Xcal$ is a family $\tund$ of crystalline Tate tensors of $\DD(\Xcal)$, such that $(\DD(\Xcal),\tund)$ is fppf locally isomorphic to $(M,\sund)$. That is, for each $(U,T,\delta) \in \CRYS(S/\ZZ_p)$ there is an fppf covering $(U_i,T_i,\delta_i)_i \to (U,T,\delta)$ such that $(\DD(\Xcal)_{(U_i,T_i,\delta_i)},\tund) \cong (M,\sund) \otimes \Ocal_{T_i}$.
 \end{definition}
 
 \begin{remark} \label{rem s structure}
 In the definitions 4.6 in \cite{kim} and  2.3.3 in \cite{HP} they demand that the family of crystalline Tate tensors satisfies the following conditions:
 \begin{assertionlist}
  \item $\Pbf := \underline\Isom ((\DD(\Xcal),\tund), (M,\sund) \otimes \Ocal_{S,\CRYS})$ is a crystal of $G$-torsors.
  \item  \'Etale locally, there exists an isomorphism $\DD(\Xcal)_S \cong M \otimes \Ocal_S$ such that $\Fil^1_{\DD(\Xcal)}$  is induced by a cocharacter in a fixed conjugacy class of $\Hom(\GG_m,G)$.
 \end{assertionlist}
 We note that the first condition is equivalent to our definition above and the second condition is satisfied if for a unique conjugacy class if $S$ is connected. Indeed, as $\Fil^1_{\DD(\Xcal)}$ is locally a direct summand, the second condition is clopen on $S$ and is satisfied for an arbitrary point of $S$ by \cite{kim} Lemma~2.5.7 (2).
 \end{remark}
 
 Given any {\BT} with $\sund$-structure, we obtain an functor
 \[
  \Gcal: \Rep_{\ZZ_p} G \to \{\textnormal{ strongly divisible filtered } F-\textnormal{crystals }\}.
 \]  
 by following Lovering's construction in \cite{lovering}. In particular, after inverting $p$, we obtain an $F$-isocrystal with $G_{\QQ_p}$-structure in the sense of Rapoport and Richartz.

  \subsection{Central leaves}
 We briefly recall Oort's notion of central leaves and extend it to {\BTTs}. We call a {\BTT} $(\Xcal^\univ, \tund)$ over an $\FF_p$-scheme $S$ \emph{geometrically fiberwise constant} if each pair of fibers of $(\Xcal^\univ, \tund)$ over points of $S$ becomes isomorphic after base change to a sufficiently large field. Under some mild conditions, it suffices to check whether $\Xcal^\univ$ is geometrically fiberwise constant.
 
 \begin{lemma} \label{lem constant BTT}
  A {\BTT} $(\Xcal, \tund)$ over a connected $k$-scheme $S$ is geometrically fiberwise constant if and only if $\Xcal$ is.
 \end{lemma}
 \begin{proof}
  Obviously $\Xcal$ is geometrically fiberwise constant if $(\Xcal, \tund)$ is. On the other hand, assume that $\Xcal$ is geometrically fiberwise constant. After replacing $S$ by an irreducible component, we may assume that $S$ is integral. By \cite{chai}~Lemma~3.3.3 there exists a surjective morphism $T \to S$, such that $\Xcal_T$ is constant, i.e.\  $\Xcal_T \cong X_T$ for a {\BT} $X$ over $k$. Thus we may assume without loss of generality that $\Xcal$ is constant and further that $S = \Spec R$ is affine and perfect. Then by \cite{RR96}~Lemma~3.9 the morphism $u: \mathbf{1} \to \DD(\Xcal)$ is constant, so in particular the isomorphism type of $(\Xcal,(u_\alpha))$ over any point of $S$ is the same.
  \end{proof}
 
 In \cite{oort04}, Oort associates to a {\BT} $\Xcal$ over an $\FF_p$-scheme $S$ a foliation of $S$, such that the restriction of $\Xcal$ to the strata is geometrically fiberwise constant. We generalize this definition to {\BTTs} as follows.
 
 \begin{definition}
  Let $(\Xcal,\tund)$ be a {\BTT} over a $k$-scheme $S$. For a fixed {\BTT} $(X_0,\tund_0)$ over in $S$ the central leaf of $(\Xcal,\tund)$ is defined as
  \[
   C_{(\Xcal,\tund)}(X_0,\tund_0) := \{ s \in S \mid (\Xcal^\univ,\tund)_{\overline{s}} \cong (X_0,\tund_0) \otimes \overline{\kappa(s)} \}
  \]
 \end{definition}
 
 \begin{proposition}
  Assume that $S$ is an excellent $\FF_p$-scheme. Then the central leaf is a closed subset of a Newton stratum.
 \end{proposition}
 \begin{proof}
  By \cite{oort04}~Prop.~2.2 this holds for \BT s without additional structure. Now the claim follows by Lemma~\ref{lem constant BTT}.
 \end{proof}
 
 \subsection{Transfer of quasi-isogenies}
 The almost product structure is based on the following construction. 
  
 \begin{construction} \label{constr isog}
  Let $S$ be an $\FF_p$-scheme and $\Xcal, \Ycal$ \BT s over $S$ such that we have an isomorphism $j: \Xcal[p^{m}] \isom \Ycal[p^{m}]$ for some integer $m$. Let $\varphi: \Ycal \to \Ycal'$ be a quasi-isogeny such that there exist integers $m_1,m_2$ with $m_1+m_2 \leq m$ such that $p^{m_1} \varphi$ and $p^{m_2} \varphi^{-1}$ are isogenies. We define the quasi-isogeny $\varphi_{\Xcal,j}$ as the concatenation
 \[
  \Xcal \stackrel{\cdot p^{-m_1}}{\longrightarrow} \Xcal \epi \bigslant{\Xcal}{j^{-1}(\ker (p^{m_1}\varphi))}.
 \]
 If it is obvious from the context which isomorphism $j$ is used we may also write $\varphi_{\Xcal,j}$ as $\varphi_\Xcal$
 \end{construction}
 
  To generalise this contruction to {\BTTs}, define the truncation $t \mod p^m$  of a crystalline Tate tensor $t$ as concatenation $\mathbf{1} \to \DD(\Xcal)^\otimes \to \DD(\Xcal[p^m])^\otimes$. Note that $\DD(\Xcal[p^m]) \cong \DD(\Xcal) \otimes \Ocal_{S, \CRYS}/p^m$ (\cite{BBM82}~Thm.~3.3.2 (ii)).
 
 \begin{proposition} \label{prop constr}
  Let $S$ be a formally smooth $\FF_p$-scheme and $(\Xcal,\tund), (\Ycal,\uund)$ be \BT s with $\sund$-structure over $S$ such that we have an isomorphism $j: \Xcal[p^{m}] \isom \Ycal[p^{m}]$ which identifies $\tund\ {\rm mod}\ p^m$ and $\uund\ {\rm mod}\ p^m$. Let $\varphi: (\Ycal,\uund) \to (\Ycal',\uund')$ be a quasi-isogeny of {\BTTs} such that there exist integers $m_1,m_2$ as above. Then $\tund' := \varphi_{\Xcal,j\, \ast} (\tund)$ is a $\sund$-structure on $\Xcal' := \image\varphi_{\Xcal,j}$. 
 \end{proposition} 
 \begin{proof}
  By construction, $\tund'$ is a family of tensors on $\DD(\Xcal')[\frac{1}{p}]$. We have to show that they factor through $\DD(\Xcal)$, are $F$-equivariant and that $(\DD(\Xcal),\tund)$ is locally isomorphic to $(M,\sund)$.
 
  We assume without loss of generality that $S = \Spec R$ is affine and denote by $A$ the universal $p$-adically complete PD-extension of a smooth lift of $R$. It suffices to check the above assertions on $A$-sections. We denote
  \begin{eqnarray*}
   \Mbf &:=& \DD(\Xcal)(A) \\
   \Mbf' &:=& \DD(\Xcal')(A) \\
   \Mbf_0 &:=& \DD(\Ycal)(A) \\
   \Mbf'_0 &:=& \DD(\Ycal')(A).
  \end{eqnarray*}
  By Remark~\ref{rem s structure}, the scheme
  \[
   P := \underline\Isom_A((\Mbf_0,\uund), (\Mbf,\tund))
  \]
  is a $G_A$-torsor and in particular smooth. Now $j$ induces an isomorphism $j_{\CRYS} \in P(A/p^m)$ which by the infinitesimal lifting criterion lifts to a section $\tilde{j}_{\CRYS} \in P(A)$. We extend $\tilde{j}_{\CRYS}$ to an isomorphism $\Mbf_0[\frac{1}{p}] \isom \Mbf[\frac{1}{p}]$.
  
 We identify $\Mbf[\frac{1}{p}] = \Mbf'[\frac{1}{p}]$ and $\Mbf_0[\frac{1}{p}] = \Mbf_0'[\frac{1}{p}]$ using the isomorphisms induced by $\varphi_{\Xcal}$ and $\varphi$. As those isogenies are compatible with crystalline Tate tensors this also identifies $\tund$ with $\tund'$ and $\uund$ with $\uund'$. In particular, $\tund'$ is $F$-equivariant.
 
  By exactness of the crystal functor (\cite{BBM82}~Cor.~4.2.8), we have
 \[
  p^{m_1} \Mbf \subset \Mbf' \subset p^{-m_2} \Mbf
 \]
 \[
  p^{m_1} \Mbf_0 \subset \Mbf_0' \subset p^{-m_2} \Mbf_0
 \]
 such that $\Mbf'$ is mapped onto $\Mbf_0'$ by $p^{m_1} \circ \tilde{j}_{\CRYS} \circ p^{-m_1} = \tilde{j}_{\CRYS}$. Thus $\tund$ factorises through $\DD(\Xcal')$ and $j_{\CRYS}$ defines an isomorphism $(\Mbf_0',\uund') \cong (\Mbf',\tund')$. In particular $(\Mbf',\tund')$ is locally isomorphic to $(M_A, \sund \otimes 1)$.
  \end{proof}
 
 \section{Leaves on deformation spaces} \label{sect leaves}
 
 In this section we mostly work with local schemes obtained by algebraising formal schemes. Let $({\rm CNLocSch}_k)$ be the full subcategory of local schemes whose objects are local schemes which are isomorphic to the spectrum of a complete Noetherian local $k$-algebra whose residue field is isomorphic to $k$. We denote an object of $({\rm CNLocSch}_k)$ as $(\Sfr, s)$ where $\Sfr$ is the scheme and $s$ the closed point of $\Sfr$. By a finite covering of $(\Sfr,s)$ we mean a surjective finite morphism $(\Sfr',s') \to (\Sfr,s)$. For any $p$-power $q$, we denote by $F_{q/k}: (\Sfr^{(q^{-1})},s) \to (\Sfr,s)$ the relative arithmetic Frobenius.
 
 \subsection{Central and isogeny leaves} \label{ss leaves}
 
  We fix an unramified local Shimura datum $(G,\bbf,-\mu)$ in the sense of \cite{kim}~Def.~2.5.10. That is, there exists a faithful $G$-representation $\Lambda$ and $b \in \bbf \cap G(W)\mu(p)G(W)$ such that $\Mbf_0 := W \otimes \Lambda^\ast$ is $b\sigma$-stable.
We choose a finite family of tensors $\sund \in \Mbf_0^\otimes$ such that $G$ is their stabilizer. We denote by $(X_0,\tund_0)$ the {\BTT} associated to the $F$-crystal with tensors $(\Mbf_0, b\sigma, \sund)$. Let $(X,\tund)$ be a {\BTT} isogenous to $(X_0,\tund_0)$; fix a quasi-isogeny $\rho: (X_0,\tund_0) \to (X,\tund)$.

 We first define the central and isogeny leaves on  the algebraisation of the mod $p$ deformation space $(\Def(X), s_{X})$  of $X$. That is, $(\Def(X),s_{X})$ represents the functor
 \begin{eqnarray*}
 ({\rm CNLocSch}_k) &\to&  ({\rm Sets})^{\rm opp} \\ 
 (\Sfr,s) &\mapsto& \{(\Xcal,\alpha) \mid \Xcal \textnormal{ is a {\BT} over } \Sfr, \alpha: X  \isom \Xcal_s \}
 \end{eqnarray*}
 Here we use that every {\BT} over the formal spectrum of a complete Noetherian $k$-algebra has a unique algebraisation by \cite{messing72}~Lemma~II.4.16 to extend the deformation problem from Artin $k$-algebras to $({\rm CNLocSch}_k)$.
 Let $(\Xcal^\univ,\alpha^{\univ})$ denote the (algebraisation  of the) universal deformation of $X_0$. The central leaf of $(\Def(X),s_{X})$ is defined as $(\Cfr(X),s_{X}) := C_{\Xcal^\univ}(X)$. We use analogous notions for the deformation space of $X_0$.
 
  Now let $\Mscr(X_0)$ be the Rapoport-Zink space associated to $X_0$ and $x \in \Mscr(X_0)(k)$ be the point corresponding to $\rho$. By the rigidity of quasi-isogenies we have a natural isomorphism between the formal neighbourhood $(\Mscr(X_0)_x^\wedge,x)$ of $x$ and $(\Def(X), s_X)$ (see for example \cite{kim}~Lemma~4.3.1). The \emph{isogeny leaf} $(\Ifr(X), s_X)$ of $(\Def(X), s_X)$ is defined as the image of the embedding \newline
  $((\Mscr(X_0)^{\rm red})_{x}^\wedge ,x) \subset (\Mscr(X_0)_x^\wedge,x) = (\Def(X), s_X)$. 
  In particular, we get a universal quasi-isogeny $\rho^\univ: X_{0, \Ifr(X)} \to \Xcal^\univ_{\Ifr(X)}$. Note that while $\rho^\univ$ does depend on the choice of $\rho$ and $X_0$, the isogeny leaf itself does not.
  
 By definition, the intersection $\Cfr(X) \cap \Ifr(X)$ parametrizes the deformations of $X$ over \emph{reduced} complete Noetherian local schemes which are geometrically fiberwise constant and isogenous to $X$. It follows from the proof of \cite{oort04}~Thm.~1.3 that these conditions imply that the deformation is in fact isomorphic to $X$. We extract the arguments from Oort's paper for the reader's convenience.

 \begin{lemma}[cf.\ \cite{oort04} \S~1.11] \label{lemma oort}
  Let $S$ be a quasi-compact reduced connected scheme over $k$ and let $\Ycal$ and $\Ycal'$ be geometrically fiberwise constant {\BT}s over $S$ together with a quasi-isogeny $\rho: \Ycal' \to \Ycal$. Let $m_1,m$ be integers such that $p^{m_1} \rho$ is an isogeny with $\ker (p^{m_1}\rho) \subset \Ycal'[p^m]$ and assume that $\Ycal'[p^m]$ is defined over $k$. Then so is $\ker (p^{m_1}\rho)$.
 \end{lemma}
 \begin{proof}
 We assume without loss of generality that $\rho$ is an isogeny and denote its kernel by $K$. By \cite{oort04} Lemma~2.9 the functor
 \[
  T \mapsto \{ H \subset Y_T \mid \rank (H/T) = \rank(K/S) \}.
 \]
 is representable by a proper $k$-scheme $Gr$. Denote by $\Hcal \to Gr$ the universal object. By \cite{oort04}~Lemma~1.10 the central leaf $C_{(Y \times Gr)/\Hcal}(Y')$ is finite for any {\BT} $Y'$ over $k$. As the morphism $S \to Gr$ induced by $K$ factors through such a central leaf and $S$ is connected, the morphism factors through a single point, i.e.\ $K$ is constant.
 \end{proof}
 
 In particular, any geometrically fiberwise constant {\BT} which is isogenous to a {\BT} which is constant (i.e.\ defined over $k$) is constant itself. Thus we obtain

 \begin{lemma} \label{lem leaf intersection}
  $\Cfr(X) \cap \Ifr(X) = \{s_X\}$.
 \end{lemma}
  
 We denote by $\Def_G (X) \subset \Def (X)$ the algebraisation of the deformation space constructed by Faltings reduced modulo $p$ in \cite{faltings99}. We will recall its explicit construction in the following subsection, here we will use a moduli description (see for example \cite{moonen98} Prop.~4.9 \cite{kim} Thm.~3.6). It is the formally smooth local subscheme of $\Def(X)$ given by the following property. A morphism $f:(\Spec k\pot{x_1,\ldots,x_N}, (x_1,\ldots,x_N)) \to (\Def(X),s_{X})$ factors through $\Def_G (X)$ if and only if there exist (necessarily unique) crystalline Tate tensors $\underline{u}$ in $f^\ast \Xcal^\univ$ which lift $\tund$. We denote by $\tund^\univ$ the universal crystalline Tate tensors of $\Xcal^\univ$ over $\Def_G (X)$.
 
\begin{remark}
  Technically, the above sources only give the moduli description is only for the formal spectra of $\Spec k\pot{x_1,\ldots,x_N}$. But as noted in \cite{HP} Rem.~2.3.5(c), the proof of \cite{dJ98} Prop.~2.4.8 shows that there exists a (necessarily unique) algebraisation of the crystalline Tate tensors. Note that they induce an $\sund$-structure. Indeed, by the rigidity of crystals it suffices to check this over the closed point, where it holds true by definition. 
\end{remark} 
 
 We define the central leaf in $\Def_G(X_0)$ as $\Cfr_G(X_0) := C_{(\Xcal_0^\univ, \tund_0^\univ)}(X_0,\tund_0)$. By Lemma~\ref{lem constant BTT} we have $\Cfr_G(X_0) = \Cfr(X_0) \cap \Def_G(X_0)$.
 
 To define the isogeny leaf $\Iscr_G(X_0)$ in $\Def_G(X)$, repeat the above construction, replacing $\Mscr(X_0)$ by the Rapoport-Zink space of Hodge type constucted in \cite{kim} (see also \cite{HP}). The universal quasi-isogeny $\rho_G^\univ$ thus obtained is in particular compatible with the crystalline Tate-tensors. As in the case of the central leaf we have $\Ifr_G(X) = \Ifr(X) \cap \Def_G(X)$, which can be used as an alternative definition if one wants to avoid using Rapoport-Zink spaces of Hodge type.
 
 \subsection{Explicit coordinates for central leaves} \label{ss central leaf}
 In this section we give an explicit description of $\Cfr_G(X_0)$ in terms of Falting's construction of the versal deformation ring of $X_0$. To make the computations bearable, we restrict ourselves to certain $b$. For example, in the case of \BT s without additional structure our restriction is equivalent to the requirement for $X_0$ to be completely slope divisible. As a consequence, we show that the central leaf is formally smooth of dimension $2\langle \rho, \nu \rangle$. This statement can be deduced in greater generality (see Proposition~\ref{prop geometry of central leaves for all}).
 
 Our construction generalises part of Chai's work, where he gave an explicit description of $\Cfr(X_0)$ as a formal subgroup of the group of successive extensions of the isoclinic subquotients of $X_0$ (\cite{chai}~\S~3, \cite{chai05}~\S~7). In particular, in this case the dimension formula and formal smoothness is due to Chai.
 
  We briefly recall Falting's construction (\cite{faltings99}~\S~7, see also \cite{moonen98}~\S~4.8, \cite{kisin10} \S~1.6, \cite{kim}~\S~3.2). Denote by $U_0 \subset \GL(\Mbf_0)$ the unipotent subgroup opposite to the unipotent induced by $\mu$ and let $\hat{U}_0$ denote the completion of $U_0$ along the identity section. Consider the following objects over the ring $A_0$ of global sections of $\hat{U}_0$.
 \begin{eqnarray*}
  \Mbf_0^\univ &=& A_0 \otimes_W \Mbf_0 \\
  \Fil^1 \Mbf_0^\univ &=& A_0 \otimes_W \Fil^1 \Mbf_0 \\
  F &:=& (u_0^{-1} \cdot (1 \otimes b)) \sigma,
 \end{eqnarray*}
 where $u_0 \in \hat{U}_0(A_0)$ is the ``universal'' section. There exists a unique connection $\nabla$ on $\Mbf_0^\univ$ which is compatible with the Frobenius $F$ and the filtration $\Fil^1 \Mbf_0^\univ$ (\cite{moonen98}~\S~4.5). Then $(\Mbf_0^\univ,\Fil^1 \Mbf_0^\univ,F,\nabla)$ is a crystalline Dieudonn\'e module, which gives rise to a {\BT} $\Xcal_0$ over $A_0/p$. Faltings showed that this is a universal deformation, thus $\Spec A_0/p \cong \Def(X_0)$ and $\Xcal_0 \cong \Xcal_0^\univ$.
 
 Denote $U := U_0 \cap G$ and let $\hat{U}$ be its completion along the identity. Let $A_G$ be the ring of global sections of $\hat{U}$ and let $(\Mbf_G^\univ,\Fil^1 \Mbf_G^\univ,F,\nabla)$ be the restriction of $(\Mbf_0^\univ,\Fil^1 \Mbf_0^\univ,F,\nabla)$ to $A_G$. Then $\Def_G(X_0) := \Spec A_G$ and $\tund_0^\univ$ is given by $\tund_0^\univ = (1 \otimes \tund_0) \subset \Mbf_G^\univ$.
 
 More explicitly, let $T \subset B \subset G$ be a maximal torus and a Borel subgroup of $G$ such that $\mu \in X_\ast(T)$. We denote by $(X^*(T),R,X_*(T),R^\vee)$ be the absolute root datum associated to $T \subset G$. For $\alpha \in R$ let $U_\alpha$ be the associated root subgroup of $G_W$ and fix an isomorphism $u_\alpha: \GG_a \isom U_\alpha$. We denote
 \[
  R_\mu = \{ \alpha \in R \mid \langle \alpha,\mu \rangle <0 \}.
 \]
  Then we get an isomorphism $\prod_{\alpha\in R_\mu} U_\alpha \cong U$ by multiplication. Note that $U$ is Abelian. so we do not need to fix an order of the factors. The $u_\alpha$ induce an isomorphism
 \[
  A_G \cong W\pot{t_\alpha \mid \alpha \in R_\mu}
 \]
 Moreover the canonical section $u \in U(A_G)$ is given by $u = \prod_{\alpha\in R_\mu} u_\alpha(t_\alpha)$. 
 
 \begin{notation} \label{not csd}
 We assume that the following assumptions hold for $b$.
  \begin{assertionlist}
   \item We require that $b = \dot{w} \sigma(\mu)(p)$ where $\dot{w} \in (\Norm_G T)(W)$. In particular, we get $\nu(g) \in X_*(T)_{\QQ}$.
   \item Let $M$ be the centralizer of $\nu$ in $G$. We require that $\dot{w} \in M$, i.e.\ $\dot{w}$ is a representative of a Weyl group element $w$ of $(M,T)$.
   \item $\nu(b)$ is anti-dominant with respect to $B$.
  \end{assertionlist}
 We can always find such an element in $\bbf \cap G(W)\mu(p)G(W)$, see for example \cite{HR}~Thm.~5.1.
 \end{notation}
 In particular, we can write $X_0 = X_0^1 \oplus \ldots \oplus X_0^m$ such that $X_0^i$ isoclinic. Let $(\Mbf_0^1,\Fil^1 \Mbf_0^1) \oplus \ldots \oplus (\Mbf_0^m,\Fil^1 \Mbf_0^m)$ the induced decomposition of $(\Mbf_0, \Fil^1\Mbf_0)$.

 Let $\Cfr$ be an irreducible component of $\Cfr_G(X_0)$. Then the restriction $\Xcal_{0 |\Sfr}^\univ$ is completely slope divisible as it is completely slope divisible over the generic point (\cite{OZ04} Prop~2.3). Its isoclinic subquotients $(\Xcal_{0 |\Cfr}^\univ)^i$ are isomorphic to $X_0^i$ as any completely slope divisible {\BT} over a Henselian ring is constant by the argument of the proof of \cite{OZ04}~Prop.~3.1.
 
 As a first step, we give thus an explicit description of the locus in $\Def_G(X_0)$ where $\Xcal_0^\univ$ allows a filtration with constant isoclinic subquotients. Let
 \begin{eqnarray*}
  R_\nu &=& \{ \alpha \in R \mid \langle \alpha,\nu \rangle <0 \} \\
  R_{\mu,\nu} &=& R_\mu \cap R_\nu
 \end{eqnarray*}
 and
 \[
  \Def_{G,\nu}(X_0) := \Spec k\pot{ t_\alpha | \alpha \in R_{\mu,\nu}}.
 \]
 
 \begin{lemma} \label{lemma coordinates for extensions}
  Let $f: (\Sfr,s) \to (\Def_G (X_0),s_{X_0})$ be a morphism corresponding to a deformation $(\Ycal,\uund; \alpha)$ over $\Sfr$. Then $\Ycal$ allows a filtration $0 = \Ycal_0 \subset \Ycal_1 \subset \Ycal_2 \subset \ldots \subset \Ycal_r = \Ycal$ with constant isoclinic subquotients if and only if $f$ factorizes over $\Def_{G,\nu}(X_0)$.
 \end{lemma}
 \begin{proof}
  Let $(\Mbf_{G,\nu},\Fil^1 M_{G,\nu},F,\nabla) \coloneqq (\Mbf_G^\univ,\Fil^1 \Mbf_G^\univ,F,\nabla) \otimes_{A_G} W\pot{t_\alpha | \alpha\in R_{\mu,\nu}}$ denote the Dieudonn\'e-module of ${\Xcal_0^{\univ}}|_{\Def_{G,\nu}(X_0)}$. The submodule $M_0^{i+1} \otimes \ldots M_0^m) \otimes_W W\pot{t_\alpha \mid \alpha \in R_{\mu,\nu}}$ is $F$-stable by construction and thus $\nabla$-stable by \cite{kisin10}~Lemma~1.4.2 (see also the errata \cite{kisin}~E.1) applied to the group $M\cdot U \subset G$. Thus the canonical projection $\Mbf_{G,\nu} \epi (\Mbf_0^1 \oplus \ldots \oplus \Mbf_0^i) \otimes_W W\pot{t_\alpha | \alpha \in R_{\mu,\nu}}$ is a morphism of Dieudonn\'e modules and thus corresponds to a filtration of ${\Xcal_0^{\univ}}|_{\Def_{G,\nu}(X_0)}$. By construction, the Dieudonn\'e module of its subquotients is isomorphic to $(\Mbf_0^i,\Fil^1 \Mbf_0^i, \dot{w}_{| \Mbf_0^i}\sigma) \otimes_W W\pot{t_\alpha|\alpha \in R_{\mu,\nu}}$ with connection uniquely determined by \cite{moonen98}~\S~4.3.2. Thus the subquotients are isomorphic to $X_0^i$ and hence indeed constant isoclinic.

  On the other hand, let $\Ycal$ be as above.  It suffices to check the assertion on infinitesimal neighbourhoods, so assume $(\Sfr,s) = (\Spec k[t]/t^n,(t))$. We denote by $\overline{f}$ and $\overline{\Ycal}$ the respective restrictions to $(\overline{\Sfr},\overline{s}) := (\Spec k[t]/t^{n-1},(t))$. By an induction argument, we may assume that $\overline{f}$ factors through $\Def_{G,\nu}(X_0)$. Let $\gamma, \overline\gamma \in U$ correspond to $f$ and $\overline{f}$, respectively. Note that $f$ factors through $\Def_{G,\nu}(X_0)$ if and only if $\gamma$ stabilises the filtration  \[ k[t]/t^n \otimes \Fil^1_{\Mbf_0^m}  \subset k[t]/t^n \otimes (\Fil^1_{\Mbf_0^m} \oplus \Fil^1_{\Mbf_0^{m-1}})  \subset \ldots \subset \Fil^1_{\Mbf_0}  \] and fixes the subquotients. 
  
  We now use Grothendieck-Messing theory on the thickening $\overline{\Ycal} \mono \Ycal$ to deduce that $f$ factorizes through $\Def_{G,\nu}(X_0)$. Fix a lift $\Ycal_0 \in \Def_{G,\nu}(X_0)(k[t]/t^n)$ and let $\gamma_0 \in U$ be the corresponding element. We identify
  \[
   \DD(\overline\Ycal)(k[t]/t^n) \cong \DD(\Ycal_0)(k[t]/t^n) \stackrel{\gamma_0}{\cong} k[t]/t^n \otimes M_0
  \]
  Using the analogous identifications for $\overline\Ycal_i$, the filtration on $\Ycal$ induces the following morphisms on the Dieudonne\'e crystal. 
  \begin{align*}
   \DD(\Ycal)(k[t]/t^n) \epi \DD(\Ycal_i)(k[t]/t^n) &\textnormal{ equals } k[t]/t^n \otimes M_0 \ \epi k[t]/t^n \otimes (M_0^1 \oplus \cdots \oplus M_0^i) \\
   \DD(\Ycal_i/\Ycal_{i+1}) \mono \DD(\Ycal_i) &\textnormal{ equals } k[t]/t^n \otimes M_0^i \mono k[t]/t^n \otimes (M_0^1 \oplus \cdots \oplus M_0^i) 
  \end{align*} 
 Note that $\Fil^1 \DD(\Ycal_0) = k[t]/t^n \otimes \Fil^1\Mbf_0 $. By an analogous argument as in \cite{kim}~Lemma~3.2.1 one deduces that $\Ycal$ corresponds via (contravariant) Grothendieck-Messing theory to the lift $\gamma\gamma_0^{-1} \cdot (k[t]/t^n \otimes \Fil^1_{\Mbf_0})$. Thus the projection $\Ycal \epi \Ycal/\Ycal_{m-1} = X_0^m \times \Sfr$ induces the embedding $k[t]/t^n \otimes \Fil^1_{\Mbf_0^m} \subset \gamma \gamma_0^{-1}(k[t]/t^n \otimes \Fil^1_{\Mbf_0})$. As $U$ is opposite to the unipotent subgroup stabilising $\Fil^\bullet_{\Mbf_0}$, this implies that $\gamma\gamma_0^{-1}$ fixes $k[t]/t^n \otimes \Fil^1_{\Mbf_0^m}$. Repeating the argument above, we obtain 
  \[ {\gamma\gamma_0^{-1}}_{|k[t]/t^n \otimes (\Fil^1_{\Mbf_0^m} \oplus \Fil^1_{\Mbf_0^{m-1}}) } \equiv 1 \mod k[t]/t^n \otimes \Fil^1_{\Mbf_0^m} .\] More generally, we obtain by induction that $\gamma\gamma_0^{-1}$ stabilizes the filtration $k[t]/t^n \otimes \Fil^1_{\Mbf_0^r}  \subset k[t]/t^n \otimes (\Fil^1_{\Mbf_0^m} \oplus \Fil^1_{\Mbf_0^{m-1}})  \subset \ldots \subset k[t]/t^n \otimes \Fil^1_{\Mbf_0}$ and acts trivially on the subquotients. As $\gamma_0$ stabilises this, thus so does $\gamma$.
 \end{proof}
 
 Let $w$ be the image of $\dot{w}$ in the Weyl group of $(G,T)$. We denote $\delta := w\sigma$. 
 
 \begin{proposition}
  $\Cfr_G(X_0) = \Spec k\pot{t_\alpha \mid \alpha \in R_C}$ where \[R_C := \{ \alpha \in R_{\mu,\nu} \mid \sum_{i=1}^{r} \langle \delta^{-i}(\alpha),\mu\rangle \leq 0 \textnormal{ for any } r> 0 \} \]
 \end{proposition} 
 
 \begin{proof}
  Let $\eta$ be a geometric point of $\Def_{G,\nu}(X_0,\tund_0)$. Then $\eta \in \Cfr_G(X)$ if and only if there exists $g \in G(W(\eta))$ such that $gb\sigma(g)^{-1} = u(\eta)^{-1}b$. As any isomorphism respects the slope filtration, $g$ must be contained in the standard parabolic associated to $M$. Denote by $N$ the unipotent radical of this parabolic, i.e. $N = \prod_{\alpha\in R_{\nu}} U_\alpha$. Writing $g = nm$ with $n \in N(W(k(\eta))$ and $m \in M(W(k(\eta)))$, we see that we must have $mg\sigma(m)^{-1} = b$. Thus we may assume without loss of generality that $g = n \in N(W(k(\eta))$. Now
  \[
   nb\sigma(n)^{-1} = u(\eta)^{-1}b \Lra u(\eta) = \Ad(b\sigma)(n) \cdot n^{-1},
  \]
  which can be expressed as system of equations as follows. Let $w$ be the image of $\dot{w}$ in the Weyl group and $\delta := w\sigma$. We write
  \[
   n = \prod_{\alpha\in R_{\nu}} u_\alpha(x_\alpha)
  \]
  for any (fixed) order of factors which satisfies the following conditions. We demand that $\langle \alpha, \nu \rangle$ is decreasing from the left-most to the right-most factor and that the factors corresponding to any $\delta$-orbit in $R_\nu$ are consecutive.  We define $\tau_\alpha := [t_\alpha(\eta)]$ for $\alpha \in R_{\mu,\nu}$ and $\tau_\alpha = 0$ otherwise, i.e. it is the $U_\alpha$-component of $u(\eta)$. 
  
  Writing both sides of $u(\eta) = \Ad(b\sigma)(n) \cdot n^{-1}$ as products of elements of the root groups in the order discussed above, we obtain the system of equations in variables $x_\alpha$ over $W(k(\eta))$
  \begin{equation} \label{eq root group}
   \tau_\alpha = - x_\alpha + C_\alpha + p^{\langle \delta^{-1}(\alpha), \mu \rangle} \epsilon_{\alpha} x_{\delta^{-1}(\alpha)}^\sigma \textnormal{ for every } \alpha \in R_\nu
  \end{equation}
  where $\epsilon_\alpha$ and $C_\alpha$ are defined below. First note that the lemma is equivalent to the claim that the above system of equations has a solution if and only if $\tau_\alpha = 0$ for $\alpha \not\in R_C$ - as we are only interested 
in whether an $n \in N(W(k))$ with $u(\eta) = \Ad(b\sigma)(n) \cdot n^{-1}$ exists, we are only interested in the existence of solutions, not their actual value.

   The constant $\epsilon_\alpha \in W^\times$ is defined by $\Ad(\dot{w})(u_\alpha(x)) = \epsilon_{w\alpha}u_{w(\alpha)}(x)$ for any $x\in W$.
   In order to cancel root group factors on both sides of the equation $u(\eta) = \Ad(b\sigma)(n) \cdot n^{-1}$, we have to commute elements from different root groups. The commutator of these element yields an additional summand for equations corresponding to higher roots. $C_\alpha$ is the sum of all such contributions to the $U_\alpha$-factor, more precisely it is obtained from commutators $[u_\beta(x_\beta), u_\gamma(\tau_\gamma)]$ with $i\beta + j\gamma = \alpha$ for some positive integers $i,j$ (cf.~\cite{springer09}~Prop.~8.2.3).
   
 Assume that there exists a system of integral solutions $(x_\alpha)$ to the above system of equations. We show by induction on $\langle \alpha, \nu \rangle$ that the following assertions hold.
 \begin{assertionlist}
  \item $\val x_\alpha \geq \sum_{k=1}^r \langle \delta^{-k}(\alpha), \mu \rangle$ for every $r \geq 0$
  \item $\tau_\alpha = 0$ if $\sum_{k=1}^r \langle \delta^{-k}(\alpha), \mu \rangle > 0$ for some $r \geq 0$ (i.e.\ if $\alpha\not\in R_C$). As $\val_p \tau_\alpha =0$ unless $\tau_\alpha = 0$, this assertion is equivalent to $\val_p \tau_\alpha \geq \sum_{k=1}^r \langle \delta^{-k}(\alpha), \mu \rangle$.
 \end{assertionlist}
  So assume this holds true for every $\alpha' \in R_\nu$ for which $\langle \alpha', \nu \rangle < \langle \alpha, \nu \rangle$. We first show that $\val_p C_\alpha \geq \sum_{k=1}^r \langle \delta^{-k}(\alpha), \mu \rangle$ for $r \geq 0$ which will basically allow us to ignore $C_\alpha$ for the rest of this proof. Indeed, the valuation of a summand corresponding to $[u_\beta(x_\beta), u_\gamma(\tau_\gamma)]$ with $i\beta + j\gamma = \alpha$ has valuation
  \[
   i\val_p(x_\beta) + j\val_p(\tau_\gamma) \geq  i \sum_{k=1}^r \langle \delta^{-k}(\beta), \mu \rangle + j \sum_{k=1}^r \langle \delta^{-k}(\gamma), \mu \rangle \\
   = \sum_{k=1}^r \langle \delta^{-k}(\alpha), \mu \rangle.
  \]
  Now we show the above claims by induction on $r$. For $r=0$ the claims are obviously true, so assume $r > 0$ and that the assertions are satisfied for $r' < r$. Then by (\ref{eq root group})
  \[ 
   \val_p(\tau_\alpha + x_\alpha) \geq \min\{C_\alpha, \langle \delta^{-1}(\alpha),\mu\rangle+\val_p(x_{\delta^{-1}(\alpha)})\} \geq \sum_{k=1}^r \langle \delta^{-k}(\alpha), \mu \rangle.
  \]

  In the case $\langle \alpha, \mu \rangle \not= -1$, we have $\tau_\alpha =0$ and the above assertions follow. So assume $\langle \alpha, \mu \rangle = -1$. Then (\ref{eq root group}) implies
  \[
   x_\alpha^\sigma = \epsilon_{\delta(\alpha)}^{-1} \cdot p^{-\langle \alpha, \mu \rangle} \cdot (\tau_{\delta(\alpha)} + x_{\delta(\alpha)} - C_{\delta(\alpha)}),
  \]
  and thus $\val_p x_\alpha \geq -\langle \alpha, \mu \rangle = 1$. Hence if $\sum_{k=1}^r \langle \delta^{-k}(\alpha), \mu \rangle > 0$, we get $\val_p \tau_\alpha >0$ and thus $\tau_\alpha=0$.

  On the other hand, assume that $\tau_\alpha = 0$ for $\alpha \not\in R_C$. We get a set of integral solutions as follows. Let $N$ be a positive integer such that $\delta^N =1$. Then inserting (\ref{eq root group}) $N-1$ times into itself, we obtain
  \begin{eqnarray*}
   x_{\alpha}  &=&   C_{\alpha} + \sum_{k=1}^{N-1} \epsilon_{\alpha,k} p^{\langle \delta^{-k}(\alpha) + \ldots + \delta^{-1}(\alpha),\mu\rangle} (C_{\delta^{-k}(\alpha)} - \tau_{\delta^{-k}(\alpha)})^{\sigma^k} \\ & & + \epsilon_{\alpha,N}p^{\langle \alpha,N\nu \rangle} x_\alpha^{\sigma^N},
  \end{eqnarray*}
  where
  \[
   \epsilon_{\alpha,k} := \prod_{l=0}^{k-1} \epsilon_{\delta^{-l}(\alpha)}^{\sigma^l}.
  \]
  Multiplying by $p^{-\langle \alpha, N\nu\rangle}$, we get
  \begin{eqnarray*}
   p^{-\langle \alpha,N\nu \rangle}\cdot x_{\alpha}  &=&   C_{\alpha} + \sum_{k=1}^{N-1} \epsilon_{\alpha,k} p^{-\langle \alpha +\ldots +\delta^{-k+1}(\alpha),\mu\rangle} (C_{\delta^{-k}(\alpha)} - \tau_{\delta^{-k}(\alpha)})^{\sigma^k} \\ & & + \epsilon_{\alpha,N} x_\alpha^{\sigma^N}.
  \end{eqnarray*}
 As all exponents are non-negative, the standard approximation argument shows that this equation has integral solutions. For every $\delta$-orbit in $R_\nu$ fix a representative $\alpha_0$ and solve the above equation for $\alpha = \alpha_0$. We obtain $x_\alpha$ for the other elements in the $\delta$-orbit by (repeatedly) solving (\ref{eq root group}) for $x_\alpha$.
 \end{proof}
 
 Now the dimension formula for $\Cfr_G(X_0)$ follows from the lemma below.
 
 \begin{lemma}
  $\# R_C = - 2\langle \rho, \nu \rangle$, where $\rho$ denotes the half-sum of positive roots in $G$.
 \end{lemma}
 \begin{proof}
  First, fix a root $\alpha_0 \in R_\nu$ and consider the set $R_C \cap \delta^\ZZ\cdot \alpha_0$. Note that $R_C \cap \delta^\ZZ\cdot \alpha_0 \subset R_{\mu,\nu} \cap \delta^\ZZ\cdot \alpha_0 = \{\alpha \in \delta^\ZZ\cdot\alpha_0\mid \langle \alpha, \mu \rangle = -1\}$. In determine the order of the complement, consider the bijection $\Phi$ given by
  \[
   \{\alpha \in \delta^\ZZ \cdot\alpha_0 \setminus R_C \mid \langle \alpha, \mu \rangle = -1\} \to \{\alpha \in \delta^\ZZ\cdot\alpha_0\mid \langle \alpha, \mu \rangle = 1\}
  \]
  \[
   \alpha \mapsto \delta^{-r(\alpha)}(\alpha)\textnormal{ with } r(\alpha) \textnormal{ minimal such that } \sum_{k=1}^{r(\alpha)} \langle \delta^{-k}(\alpha), \mu \rangle > 0.
  \]
 To prove that $\Phi$ is a bijection, we show that \[\Psi: \alpha \mapsto \delta^{r'(\alpha)} \textnormal{ with } r'(\alpha) \textnormal{ minimal such that } \sum_{k=0}^{r'(\alpha)} \langle \delta^{k}(\alpha), \mu \rangle = 0\] is its inverse. Indeed, we have $\Psi \circ \Phi (\alpha) = \alpha$, since 
 \[
  \langle\delta^{-r(\alpha)}(\alpha) + \ldots + \alpha, \mu\rangle = \langle\delta^{-r(\alpha)}(\alpha) + \ldots + \delta^{-1}(\alpha), \mu\rangle + \langle \alpha, \mu\rangle = 1 +(-1) = 0
 \]
 and if there were $r' < r(\alpha)$ such that $\sum_{k=0}^{r'} \langle \delta^{-r(\alpha)+k}(\alpha), \mu \rangle = 0$, we would obtain \newline $\sum_{k=1}^{r(\alpha)-r'} \langle \delta^{-k}(\alpha), \mu \rangle > 0$, contradicting the definition of $r(\alpha)$. A analogous argument shows $\Phi \circ \Psi = \id$. Now
  \begin{eqnarray*}
   \# R_C \cap \delta^\ZZ\cdot \alpha_0 &=& \#\{\alpha \in \delta^\ZZ\cdot\alpha_0\mid \langle \alpha, \mu \rangle = -1\} - \{\alpha \in \delta^\ZZ\cdot\alpha_0\mid \langle \alpha, \mu \rangle = 1\} \\
   &=& -\sum_{\alpha \in \delta^\ZZ\cdot\alpha_0} \langle \alpha, \mu \rangle \\
   &=& -\sum_{\alpha \in \delta^\ZZ\cdot\alpha_0} \langle \alpha, \nu \rangle.
  \end{eqnarray*}
  By summing over all orbits, we obtain
  \begin{eqnarray*}
   \# R_C &=& - \sum_{\alpha \in R_\nu} \langle \alpha, \nu \rangle \\
   &=& -\sum_{\alpha >0} \langle \alpha, \nu \rangle \\
   &=& -\langle 2\rho, \nu \rangle
  \end{eqnarray*}
 \end{proof} 
 
 As a consequence, we obtain the following proposition. 
 \begin{proposition} \label{prop geometry of central leaves}
  If the assumptions of Notation~\ref{not csd} are satisfied, $\Cfr_G(X_0)$ is formally smooth of dimension $2\langle \rho,\overline\nu(b) \rangle$.
 \end{proposition}

 The conditions we imposed on $(X_0,\tund_0)$ are rather restrictive. However, anticipating results from later sections of this article, we can show that the above Proposition holds in greater generality. 
 
 \begin{proposition} \label{prop geometry of central leaves for all}
 Let $(X_0,\tund_0)$ be an arbitrary {\BT} with $\sund$-structure over $k$.
  \begin{subenv}
   \item We have $\dim \Cfr_G(X_0) = 2\langle \rho,\overline\nu(b) \rangle$.
   \item If $(G,\bbf,-\mu)$ is induced by a Shimura datum of Hodge type, $\Cfr_G(X_0)$ is formally smooth.
  \end{subenv}
 \end{proposition}
 \begin{proof}
  As a consequence of Proposition~\ref{prop almost product local}~(2), the dimension of the central leaf only depends on the isogeny class of $(X_0,\tund_0)$. Thus the dimension formula follows from  Proposition~\ref{prop geometry of central leaves}.
   If $(G,b,-\mu)$ is induced by a Shimura datum of Hodge type,
 $\Def_G(X_0)$ is canonically isomorphic the formal neighbourhood of  a point in a Shimura variety of Hodge type (\cite{kisin}~Prop.~1.3.10). It is shown in \cite{lovering}~Thm.~3.3.12, that this identifies $\Cfr_G(X)$ with the formal neighbourhood of the central leaf in the Shimura variety. But central leaves in Shimura varieties are smooth, as the formal neighbourhoods of its closed points are isomorphic to each other.
 \end{proof}
  
 \subsection{Trivialisation of $p^m$-torsion} 
 
  By \cite{oort04}~Thm.~1.3 there exists a finite surjective map of schemes $\widetilde\Cfr(X_0) \to \Cfr(X_0)$ such that the pullback of $\Xcal_0^\univ[p^m]$ to $\widetilde\Cfr(X_0)$ is constant. As we will have certain compatibility conditions, we cannot use this theorem directly. Instead, we follow Oort's construction and explain why all criteria are met. Starting point is the following observation (see \cite{mantovan04}~Lemma~4.1, \cite{OZ04}~Prop.~1.3 for proofs).
  
 If $\Xcal$ is a completely slope divisible {\BT} with isoclinic subquotients $\Xcal^i$, then there exists a canonical isomorphism
 \[
  \Xcal[p^m]^{(p^N)} \cong \Xcal^1[p^m]^{(p^N)} \oplus \ldots \oplus \Xcal^n[p^m]^{(p^N)}
 \]
 for $m>0$ and $N$ big enough (depending on $m$). As these isomorphism commute with base change, we obtain a family of isomorphisms
 \[
  i_{N,m}:  F_{p^N/k}^\ast \Xcal^1[p^m] \oplus \ldots \oplus F_{p^N/k}^\ast \Xcal^n[p^m] \isom F_{p^N/k}^\ast \Xcal[p^m].
 \]
 The $\iota_{N,m}$ satisfy the obvious commutativity relations, allowing us to define
 \[
  i_\infty := \underrightarrow{\lim}_m\, \underleftarrow{\lim}_N\, \iota_{N,m} : F_{p^\infty/k}^\ast (\Xcal^1 \oplus \ldots \oplus \Xcal^n) \isom F_{p^\infty/k}^\ast \Xcal.
 \]
 
 \begin{lemma} \label{lemma covering}
  Assume that the assumptions of Notation~\ref{not csd} are satisfied.
  \begin{subenv}
   \item There exists a canonical isomorphism $\iota_{N,m}:  X_0[p^m] \times\Cfr(X_0)^{(p^{-N})} \isom F_{p^{N}/k}^\ast \Xcal_0^\univ[p^m]$ for $m > 0$ and $N$ big enough (depending on $m$) such that the obvious compatibility criteria are satisfied for varying $m,N$ and such that $\iota_{N,m, s_0} = \alpha^\univ$.
   \item $\iota_{N,m}$  identifies $\tund_0 \mod p^m$ and $\tund_0^{\univ} \mod p^m$.  
  \end{subenv}
 \end{lemma}
 \begin{proof}
  Recall that $\Xcal_{|\Cfr(X_0)}^\univ$ is completely slope divisible and that $\alpha^\univ$ extends to a (necessarily unique) isomorphism $j$ sending $X_0 \times \Cfr(X_0)$ to the direct sum of the isoclinic subquotients of $\Xcal^\univ_{0\,|\Cfr(X)}$. Now let for $N$ big enough $\iota_{N,m} = i_{N,m} \circ j_{|X[p^m] \times \Cfr(X)}$.
  
  It suffices to show the second assertion after pulling back to  $\Cfr_G(X_0)^{(p^{-\infty})}$.
  By construction, $\Cfr_G(X_0)^{(p^{-\infty})}$ is perfect and there exists an isomorphism $\iota_\infty: X_0 \times \Cfr_G(X_0)^{(p^{-\infty})} \isom \Xcal^\univ_{0\, \Cfr_G(X_0)^{(p^{-\infty})}} $. Now $\iota_\infty(\tund_0 \otimes 1)$ and $\tund^\univ_0 \otimes 1$ are equal over $s_{X_0}$ by construction, thus they are equal over $\Cfr(X_0)^{(p^\infty)}$ by \cite{RR96}~Lemma~3.5. In particular, they are equal modulo $p^m$.
 \end{proof}  

 Similar as in the previous chapter, the assertions of Notation~\ref{not csd} are not necessary.
 
 \begin{lemma} \label{lemma covering for all}
  Let $(X_0,\tund_0)$ a {\BT} with $\sund$-structure over $k$.
   \begin{subenv}
   \item There exists a canonical isomorphism $\iota_{N,m}:  X_0[p^m] \times\Cfr(X_0)^{(p^{-N})} \isom F_{p^{N}/k}^\ast \Xcal_0^\univ[p^m]$ for $m > 0$ and $N$ big enough (depending on $m$) such that the obvious compatibility criteria are satisfied for varying $m,N$ and such that $\iota_{N,m, s_0} = \alpha^\univ$.
   \item For any $f:(\Sfr,s) \to \Cfr_G(X_0)$ with $(\Sfr,s)$ formally smooth $\iota_{N,m}$ identifies $f^\ast \tund_0 \mod p^m$ and $f^\ast \tund_0^{\univ} \mod p^m$.  
  \end{subenv}
 \end{lemma}
 \begin{proof}
  By \cite{OZ04} Prop.~3.1 we can find a quasi-isogeny $\Ycal \to 
 \Xcal^\univ_{0\, |\Cfr_G(X_0)}$ with $\Ycal$ slope divisible. Arguing as in the proof above, we obtain a family of canonical isomorphisms $\iota_{N,m}' :Y[p^m] \times\Cfr(X_0)^{(p^{-N})} \isom F_{p^{N}/k}^\ast \Ycal[p^m]$ for a {\BT} $Y$ over $k$. If $\ker \rho \subset \Ycal[m']$, then $\iota_{N,m}'$ induces the isomorphism $\iota_{N,m-m'}$ with the wanted properties. As the second part did not use any assumptions on $(X_0,\tund_0)$, the proof is literally the same as above.
 \end{proof}
 
 \section{The almost product structure for deformation spaces} \label{sect almost product local}

 \subsection{Construction of the almost product structure}
 Now fix $m_1,m_2$ big enough such that $p^{m_1}\rho^\univ$ and $p^{m_2}\rho^{\univ\, -1}$ are isogenies and let $m = m_1+m_2$. We fix a finite covering $(\widetilde\Cfr(X_0), \tilde{s}_{X_0}) \to (\Cfr(X_0),s_{X_0})$ such that there exists an isomorphism $\iota: X_0[p^m] \times \widetilde\Cfr(X_0) \isom \Xcal_0^\univ[p^m]_{\widetilde\Cfr(X_0)}$ which is compatible with $\alpha^\univ$ and which can be lifted to $p^{m'}$-torsion points after pulling back to some finite covering for any $m' > m$.
 
 Consider the product $ (\Pfr,p) := (\widetilde\Cfr(X_0),\tilde{s}_{X_0}) \times (\Ifr(X),s_X)$ with canonical projections $\pr_1, \pr_2$. Denote by $\tilde\Xcal_0 := \pr_1^\ast \Xcal^\univ_0$ and by $\tilde\Xcal$ the image of $\tilde\rho := (\pr_2^\ast \rho^\univ)_{\tilde\Xcal_0}$. Then $\tilde\Xcal$ induces a morphism $\pi: (\Pfr,p) \to (\Mscr(X_0)_x^\wedge, x) \cong (\Def (X), s_X)$. It is easy to see that $\pi$ factorises through $\Nfr(X)$. 
 
 Of course $\pi$ depends on the choice of $\widetilde\Cfr(X_0)$, but for two choices $\widetilde\Cfr(X_0)$ and $\widetilde\Cfr(X_0)'$ such that $\widetilde\Cfr(X_0) \to \Cfr(X_0)$ factorises through $\widetilde\Cfr(X_0)'$ and the trivialisations of the $p^m$-torsion points are compatible we get a commutative diagram
 \begin{center}
  \begin{tikzcd}
   &  (\widetilde\Cfr(X_0),\tilde{s}_{X_0}) \times (\Ifr(X),s_X) \arrow{dd} \arrow{dddl} \arrow{dddr}{\pi} & \\ \\
   & (\widetilde\Cfr(X_0)',\tilde{s}_{X_0}') \times (\Ifr(X),s_X) \arrow{dl} \arrow{dr}{\pi'} & \\
   (\Cfr(X_0),s_{X_0}) & & (\Def(X),s_X)
  \end{tikzcd}
 \end{center}
 We will use this fact frequently to argue that we can enlarge $\widetilde\Cfr(X_0)$ without losing generality. 
 
  \begin{remark}
  In the case where $X_0$ is completely slope divisible and the coverings are as in Lemma~\ref{lemma covering}, $\pi$ is the restriction of Mantovan's morphism $\pi_N$ in (\cite{mantovan04}) to formal neighbourhoods.
 \end{remark} 
 
 \begin{proposition} \label{prop almost product local}
  Let $\pi$ as above.
  \begin{subenv}
   \item For a geometric point $\eta = (\eta_1,\eta_2) \in \Pfr$ the isomorphism class of $\tilde\Xcal^\univ_\eta$ coincides with the isomorphism class of $X_0/(\ker (p^m \rho^\univ))_{\eta_2}$.
   \item $\pi^{-1}(\Cfr(X)) = \widetilde\Cfr(X_0) \times \{s_X\}$ and $\pi^{-1}(\Ifr(X)) = \{\tilde{s}_{X_0}\} \times \Ifr(X)$.
   \item $\pi$ is finite and surjective.
   \item Let $(\Sfr,s) \in ({\rm CNLocSch})$ arbitrary and $P_1,P_2$ two $(\Sfr,s)$-valued points of $(\Pfr,p)$. We denote $(\Ycal_i, \alpha_i) := P_i^\ast (\widetilde\Xcal_0,\tilde\alpha)$  and $\rho_i := P_i^\ast\tilde\rho$. One has $\pi(P_1) = \pi(P_2)$ iff $(\Ycal_1,\alpha_1) = (\Ycal_2,\alpha_2)$ and $\rho_{1, \Ycal_1} = \rho_{2, \Ycal_2}$.
  \end{subenv}
 \end{proposition} 

 \begin{proof}
  For any natural number $N$ we may enlarge $\Cfr(X_0)$ to assume that the pullback of $\Xcal^\univ_{0}[p^{m+N}]$ and $X_0[p^{m+N}]$ to $\widetilde\Cfr(X_0)$ are isomorphic and thus the group schemes of $p^N$ torsion points of $\tilde\Xcal^\univ_\eta$ and $X_0/(\ker (p^m \rho^\univ))_{\eta_2}$ are isomorphic. By \cite{oort04}~Cor.~1.7 this implies that the \BT s are also isomorphic if $N$ is big enough, proving (1). 
  
  Now (1) implies that $\pi(\eta) \in \Cfr(X)$ if and only if $\eta_2 \in \Cfr(X)$, which by Lemma~\ref{lem leaf intersection} is equivalent to $\eta_2 = s_X$; this shows the first assertion of (2). Certainly $ \{\tilde{s}_{X_0}\} \times \Ifr(X) \subset \pi^{-1}(\Ifr(X))$. To show equality, it suffices to show that $\Tfr := \pi^{-1}(\Ifr(X)) \cap \widetilde\Cfr(X_0) \times \{s_X\}$ equals $\{p\}$, as $\pi^{-1}(\Ifr(X))$ is closed under specialization. Indeed, we have
  \[
   \pi( \Tfr ) \subset \Ifr(X) \cap \pi(\widetilde\Cfr(X_0) \times \{s_X\}) = \Ifr(X) \cap \Cfr(X) = \{p\},
  \]
  which implies that $\tilde\Xcal^\univ_{0,\Tfr} \cong X_0$. Hence $\Tfr \subset \pr_1^{-1}(s_{X_0}) = \{p\}$.  
  
  As consequence of (2) we see that $\pi^{-1}(s_X) = \{p\}$, in particular the restriction $\pi_{|\pi^{-1}(s_X)}$ is finite, as $(\Pfr,p)$ is topologically of finite type. By \cite{EGA1}~Cor.~7.4.3 this implies that $\pi$ is finite. To show surjectivity, fix a morphism $f: (\Sfr,s) \to (\Nfr(X),s_X)$ with $(\Sfr,s)$ integral. We will show that after replacing $(\Sfr,s)$ by a finite covering, this  $(\Sfr,s)$-point is contained in the image of $\pi$. For this we need to construct a quasi-isogeny $\varphi: \Ycal \to \Xcal^\univ_\Sfr$  where the geometric fibres of $\Ycal$ are isomorphic to $X_0$ and $\varphi_s = \rho$ after fixing an isomorphism $X_0 \cong \Ycal_s$. Now $\Ycal$ corresponds to an $(\Sfr,s)$-valued point $(\Cfr(X_0),s_{X_0})$ and after replacing $(\Sfr,s)$ by a finite covering, we may assume that there exists a lift $\tilde{y}$ to $(\widetilde{\Cfr}(X_0),s_{X_0})$. Then $\varphi$ is induced by an isogeny $\varphi_0: X \to X_{\Sfr}^\univ$ and $\pi(\tilde{y},\varphi_0) = f$.
  
   Following the construction in \S 2 of \cite{OZ04}, we obtain a quasi-isogeny $\psi: \Xcal^\univ_\Sfr \to \Ycal$ where the geometric fibres of $\Ycal$ are isomorphic to $X_0$.  We sketch the construction of $\psi$ for the reader's convenience.
  \begin{bulletlist}
   \item Let $\eta$ be the generic point of $\Sfr$. Since $\Xcal_\eta \otimes \overline{k(\eta)}$ is isogenous to $X_0 \otimes \overline{k(\eta)}$, there exists a finite extension $K$ of $k(\eta)$ such  that $\Xcal^\univ_\eta \otimes K$ is isogenous to $X_0 \otimes K$. By replacing $\Sfr$ by its normalisation in $K$, we may assume $K = k(\eta)$.
   \item Fix an isogeny $\psi_\eta: \Xcal^\univ_\eta \to X \otimes k(\eta)$. Let $U \subset \Sfr$ the flat locus of the closure of $\ker \psi_\eta$. Let $\psi_U: \Xcal^\univ_U \to \Xcal^\univ_U/\overline{\ker \psi_\eta}$.
   \item Let $\Mscr$ the scheme representing quasi-isogenies from $\Xcal^\univ_{\Sfr}$ of degree $d := \deg \psi_\eta$. Then $\psi_U$ defines a section $U \to \Mscr_U$. Taking the closure of its image we obtain a continuation of $\psi_U$ to a projective $\Sfr$-scheme $\tilde{S}$ (cf.\ \cite{OZ04}~Lemma~2.4).
   \item This isogeny descends to $\Sfr$ by the same argument as in \cite{OZ04}~Prop.~2.7.
  \end{bulletlist}
  Since $\Ycal$ is generically geometrically isomorphic to $X_0$ and has constant Newton polygon, it is geometrically fiberwise isomorphic to $X_0$ by \cite{oort04}~Prop.~2.2. Choose an isomorphism $\alpha_0: \Ycal_s \cong X_0$ and let $\rho':= \alpha_0 \circ \psi_s \circ \alpha^{\univ \, -1}$. We need to tweak $\psi$ such that $\rho' = \rho^{-1}$. So let $j:= \rho' \circ \rho$ and choose $m_1',m_2'$ such that $p^{m_1'}j$ and $p^{m_2'}j^{-1}$ are isogenies and set $m' = m_1'+m_2'$. After replacing $(\Sfr,s)$ by a finite covering, we may assume that $\Ycal[p^{m'}]$ is constant and choose an isomorphism $\Ycal[p^{m'}] \cong X_0[p^{m'}]$ which can be extended to truncations of any level after pulling back to a finite covering. Now $\varphi := \psi \circ j_\Ycal$ satisfies the conditions above.
  
  The ``if'' assertion in (4) is obvious. Now assume that $\pi(P_1) = \pi(P_2)$. We denote $\varphi := \rho_2^{-1} \circ \rho_1: \Ycal_1 \to \Ycal_2$. It suffices to show that $\varphi$ is an isomorphism. Fix $n_1$ such that $p^{n_1}\varphi$ is an isogeny and $n$ such that $\ker(p^{n_1} \varphi) \subset \Ycal_1[p^n]$. Replacing $\Sfr$ by an fppf-covering if necessary, we may assume that there is an isomorphism $\iota': X_0[p^n] \times \Sfr \to \Ycal_1[p^n]$. Then $\iota'(\ker(p^n_1 \varphi))$ is a constant subscheme of $X_0[p^n] \times \Sfr$ by Lemma~\ref{lemma oort}. Since $\varphi$ is an isomorphism over the special point, it follows that $ \iota'(\ker(p^{n_1} \varphi)) = X_0[p^{n_1}] \times \Sfr$; hence $\varphi$ is an isomorphism.
 \end{proof} 
 
 While $\Nfr(X)$ and $\Cfr(X_0) \times \Ifr(X)$ will in general not be isomorphic, we can deduce from the above proposition that they are inseperable forms of each other. Let us make this more precise.
 
 In the case where $\widetilde{\Cfr}(X_0) = \Cfr(X_0)^{(p^{-N})}$ and $\iota$ is defind as in Lemma~\ref{lemma covering for all}, denote $\pi_N := \pi$. By construction the $\pi_N$ induce the same morphism $\pi_\infty: \Pfr^{(p^{-\infty})} \to \Nfr(X)^{(p^{-\infty})}$.
 
 \begin{corollary} \label{cor perfect product local}
  $\pi_\infty$ is an isomorphism
 \end{corollary} 
 \begin{proof}
 By \cite{BS}~Lemma~3.8 any universal homeomorphism of perfect schemes is an isomorphism.  As $\pi_\infty$ is integral and surjective, it suffices to show that is a monomorphism in the category of perfect schemes over $k$, which is canonically equivalent to the localisation of the category of $k$-schemes obtained by inverting $F_{p/k}$. A morphism is a monomorphism iff its diagonal morphism is an isomorphism. As $\pi_N$ is finite, its diagonal morphisms in the categories $({\rm Sch_k})$ and $({\rm LocNoeth_k})$ coincide. Thus, it suffices to prove that $\pi_N$ is a monomorphism in $({\rm LocNoeth_k})[F_{p/k}^{-1}]$.
 
 Let $P_1, P_2$ be as in part (4) of above proposition. Now $(\Ycal_1,\alpha_1) = (\Ycal_2,\alpha_2)$ is equivalent to $F_{p^N/k}\circ \pr_1 \circ P_1 = F_{p^N/k} \circ \pr_1 \circ P_2$. Moreover the endomorphism $\iota_{P_2} \circ \iota_{P_1}$ of $\Ycal_1[p^m]$ becomes trivial after pullback to $\Sfr^{(p^{-\infty})}$ as it can be lifted to an endomorphism of $\Ycal_{1\, \Sfr^{(p^{-\infty})}}$. Thus $\rho_1 = \rho_2$ implies that $F_{p/k}^{N'}\circ \pr_2 \circ P_1 = F_{p/k}^{N'}\circ \pr_2 \circ P_2$ for $N'$ big enough. Hence for $n =\max\{N,N'\}$ we get $F_{p/k}^{n\, \ast} P_1 = F_{p/k}^{n\, \ast} P_2$. 
 \end{proof} 
 
 \subsection{Restriction to $\Def_G$}
 
 Now assume that $\iota$ identifies the pullbacks of $\tund_0^\univ {\rm mod} \,p^m$ and $\tund_0\, {\rm mod}\, p^m$ to $\widetilde\Cfr_G(X_0)$. Moreover we assume that for any $m' > m$ we can pull back to some finite covering where $\iota$ can be lifted to an isomorphism of $p^{m'}$-torsion points which is compatible with the $\sund$-structure.
 
  We denote $(\Pfr_G,p) := (\widetilde\Cfr_G(X_0),s_{X_0}) \times (\Ifr_G(X),s_X) \subset (\Pfr,p)$ and let $\pi_G$ be the restriction of $\pi$ to $(\Pfr_G,p)$.
 
 \begin{proposition} \label{prop almost product local hodge}
  The image of $\pi_G$ is $\Nfr_G(X)$.
 \end{proposition} 
 \begin{proof}
  It follows from a simple topological argument that it suffices to check the claim on points of dimension $1$, as they are dense.
    
  So let $\eta \in \Pfr_G$ of dimension one and let $P \in (\Pfr,p)(k\pot{t},(t))$ be the normalisation of its closure. Then $\pi_G(P)$ factors through $\Nfr_G(X)$ by Proposition~\ref{prop constr}, thus $\pi_G(\eta) \in \Nfr(G)(X)$. On the other hand let $f: (\Spec k\pot{t}, (t)) \to (\Nfr_G, s_X)$. It suffices to show that the preimage of $f$ constructed in the proof of Proposition~\ref{prop almost product local} lies in $(\Pfr_G,p)$. First note that any finite extension of $k\pot{t}$ we take during the construction may be chosen such that it is again isomorphic to $ k\pot{t}$.
  
  In the first step of the construction we require that $\psi_\eta$ is compatible with crystalline Tate tensors. Then $\psi_{\ast} (\tund^\univ)$ factors over $\DD(\Ycal)^\otimes$ as this holds true in the generic fiber by construction. After replacing $\psi$ and $\Ycal$ such that we get the correct isogeny in the special fibre this still holds true by Proposition~\ref{prop constr} if we ensure that $(\Ycal[p^m], \psi_{\ast} (\tund^\univ) \mod p^m)$ is constant. This can be shown as in Lemma~\ref{lemma covering}.
  
  Thus the $(\Spec k\pot{t}, (t))$-valued point of $(\Cfr(X_0), s_{X_0})$ which is induced by $\Ycal$ is a point in  $(\Cfr_G(X_0), s_{X_0})$. Hence its lift $\tilde{y}$ is contained in $(\widetilde\Cfr_G(X_0), s_{X_0})(\Spec k\pot{t}, (t))$. Similarly, $\varphi_0$ defines a point in $(\Ifr_G(X), s_X)$ as it respects the crystalline Tate-tensors. Thus the preimage $(y,\varphi_{X_0})$ is indeed a point of $(\Pfr_G,p)$.
 \end{proof}  
 
 The following properties of $\pi_G$ are a direct consequence of the analogous properties of $\pi$.
 
 \begin{proposition} \label{prop properties almost product local hodge}
  Let $\pi_G$ as above.
  \begin{subenv}
   \item $\pi_G$ is finite and surjective. In the situation of Lemma~\ref{lemma covering for all} the perfection of $\pi_G$ is an isomorphism.
   \item $\pi_G^{-1}(\Ifr_G(X)) = \{s_{X_0}\} \times \Ifr_G(X)$ and $\pi_G^{-1}(\Cfr_G(X)) = \Cfr_G(X_0) \times \{s_X\}$. 
  \end{subenv}
 \end{proposition} 
 
 \section{Dimension formulas} \label{sect dimension formulas}
  
 \subsection{The Newton stratification on deformation spaces}  
  
  We keep the notation of the previous subsection and moreover denote by $b_0$ the isogeny class of $(X,\tund)$. As a consequence of Proposition~\ref{prop almost product local hodge} we obtain the equality
 \[
  \dim \Nfr_G(X) = \dim \Cfr_G(X) + \dim \Ifr_G(X).
 \] 
 Using the interpretation of $\Ifr_G(X)$ as formal neighbourhood of a point in the reduced subscheme associated to the Rapoport-Zink space of Hodge type, we know that the dimension of $\Ifr_G(X)$ is less or equal than the dimension of the Rapoport-Zink space. Thus by \cite{zhu},
 \[ 
  \dim \Ifr_G(X) \leq \langle \rho, \mu-\overline{\nu}(b_0) \rangle - \frac{1}{2} \defect_G(b_0).
 \]
 As we proved in section \ref{ss central leaf}, the dimension of the central leaf equals $2\langle \rho, \overline{\nu}(b_0) \rangle$. Thus the above inequality is equivalent to
 \[
  \dim \Nfr_G(X) \leq \langle \rho, \mu+\overline\nu(b_0) \rangle - \frac{1}{2} \defect_G(b_0).
 \] 
 This inequality can be expressed in more natural terms. For two isogeny classes $b' \leq b$, denote by $\ell[b',b]$ the maximal length of a chain $b' < b_1 < \cdots < b$ of isogeny classes. Using the corrected formula of \cite{chai00}~Thm.~7.4 (see for example \cite{hamacher15b}~Prop.~3.11, \cite{viehmann}~Thm.~3.4), the above inequality is equivalent to
 \begin{equation} \label{ineq Newton}
  \codim \Nfr_G(X) \geq \ell[b_0,b_{\rm max}],
 \end{equation}
 where $b_{\rm max} = [\mu(p)]$.
 
 Now the dimension formula and closure relations of Newton strata in $\Def_G(X)$ follow from this inequality and the purity of the Newton stratification by a formal argument given in \cite{viehmann}~Lemma~5.12.\footnote{Actually, Viehmann proves it under the assumption of strong purity. But her proof only uses topological strong purity.}
 
 \begin{proposition} \label{prop local Newton}
  For any $b_0 \leq b \leq b_{\max}$
  \begin{subenv}
   \item $\Def_G(X)^{\leq b}$ is of pure codimension $\ell[b,b_{\rm max}]$ in $\Def_G(X)$.
   \item $\Def_G(X)^{\leq b}$ is the closure of $\Def_G(X)^b$.
  \end{subenv}
  In particular, $\Def_G(X)^b$ is non-empty.
 \end{proposition} 
 
 \subsection{Global results}
 We obtain the following global results as easy corollaries to our considerations above.
 
 \begin{corollary}
  The underlying reduced subscheme of a Rapoport-Zink space of Hodge type $\Mscr_{G,\mu}(b)$ is equidimensional.
 \end{corollary} 
 \begin{proof}
  The formal neighbourhood of a closed point in the underlying reduced subscheme of a Rapoport-Zink space of Hodge type is isomorphic to an isogeny leaf $\Ifr_G(X)$(cf.~section~\ref{ss leaves}). Hence it is of dimension
 \[
  \dim \Ifr_G(X) = \dim \Nfr_G(X,\tund) - \dim \Cfr_G(X) = \langle \rho, \mu-\nu \rangle - \frac{1}{2} \defect_G(b),
 \]  
 which coincides with the dimension of underlying reduced subscheme of the Rapoport-Zink space.
 \end{proof}
 
 Let $\Sscr_{G,0}$ be the special fibre of the canonical model of a Shimura variety of Hodge type. 
 By construction, the formal neighbourhood of any point $x \in \Sscr_{G,0}(\FFbar_p)$ is canonically isomorphic to some deformation space $\Def_G(\Acal_x[p^\infty])$ and the isomorphism preserves the Newton stratification (see \cite{lovering}~Thm.~3.3.12). Thus we obtain the following corollary.
 
 \begin{corollary}
  Let $b \in B(G,\mu)$.
  \begin{subenv}
   \item $\Sscr_{G,0}^{\leq b}$ is of pure codimension $\ell[b,b_{\rm max}]$ in $\Sscr_{G,0}$ .
   \item $\Sscr_{G,0}^{\leq b}$ is the closure of $\Sscr_{G,0}^b$.
  \end{subenv}
 \end{corollary} 

 \providecommand{\bysame}{\leavevmode\hbox to3em{\hrulefill}\thinspace}
\providecommand{\MR}{\relax\ifhmode\unskip\space\fi MR }
\providecommand{\MRhref}[2]{%
  \href{http://www.ams.org/mathscinet-getitem?mr=#1}{#2}
}
\providecommand{\href}[2]{#2}

 \end{document}